\documentclass[12pt,a4paper]{amsart}
\usepackage{url,color}
\usepackage{amsmath,amsthm,amsfonts,amssymb,latexsym,accents,yhmath, tikz, tikz-cd, float}
\usepackage{bbm}
\usepackage{todonotes}
\usepackage{calrsfs}
\usepackage{url}
\usepackage[colorlinks=true, urlcolor=black, citecolor=black,
linkcolor=black, hyperfootnotes=true]{hyperref}

\theoremstyle{definition}

\DeclareMathAlphabet{\pazocal}{OMS}{zplm}{m}{n}

\newcommand{\miller}{\mathit{m}}

\newcommand{\IQ}{\mathbb Q}

\newcommand{\IP}{\mathbb P}

\newcommand{\D}{\mathcal{D}}

\newcommand{\bigvid}{\wideparen{\ \ }}

\newcommand{\uhr}{\upharpoonright}

\newcommand{\Split}{\mathrm{Split}}

\newcommand{\hot}{\mathfrak}

\newcommand{\nothing}[1]{}

\newcommand{\rng}{\mathrm{rng}}
\newcommand{\scsr}{\mathrm{scsr}}

\newcommand\tsup[2][2]{%
 \def\useanchorwidth{T}%
  \ifnum#1>1%
    \stackon[-1pt]{\tsup[\numexpr#1-1\relax]{#2}}{\hspace{1pt}\scriptstyle\sim}%
  \else%
    \stackon[.5pt]{#2}{\hspace{1pt}\scriptstyle\sim}%
  \fi%
}

\newcommand{\nc}{\newcommand}

\nc{\cO}{\mathcal{O}}

\nc{\ram}{\mathsf{Ramsey}}
\nc{\soo}{\mathsf{S}_1(\cO,\cO)}
\nc{\swl}{\mathsf{S}_1(\Om,\Lambda)}
\nc{\goo}{\gone(\cO,\cO)}
\nc{\gwo}{\gone(\Om,\cO)}
\nc{\gwl}{\gone(\Om,\Lambda)}
\nc{\soox}[1]{\mathsf{S}_1(\cO(#1),\cO(#1))}
\nc{\swlx}[1]{\mathsf{S}_1(\Om(#1),\Lambda(#1))}
\nc{\goox}[1]{\gone(\cO(#1),\cO(#1))}
\nc{\gwox}[1]{\gone(\Om(#1),\cO(#1))}
\nc{\gwlx}[1]{\gone(\Om(#1),\Lambda(#1))}

\nc{\sfinoo}{\mathsf{S}_{\mathrm{fin}}(\cO,\cO)}
\nc{\sfinwl}{\mathsf{S}_{\mathrm{fin}}(\Om,\Lambda)}
\nc{\sfinww}[1]{\mathsf{S}_{\mathrm{fin}}(\Om(#1),\Om(#1))}
\nc{\gfinoo}{\gfin(\cO,\cO)}
\nc{\gfinwo}{\gfin(\Om,\cO)}
\nc{\gfinwl}{\gfin(\Om,\Lambda)}
\nc{\sfinoox}[1]{\mathsf{S}_{\mathrm{fin}}(\cO(#1),\cO(#1))}
\nc{\sfinwlx}[1]{\mathsf{S}_{\mathrm{fin}}(\Om(#1),\Lambda(#1))}
\nc{\sfinwwx}[1]{\mathsf{S}_{\mathrm{fin}}(\Om(#1),\Om(#1))}
\nc{\gfinoox}[1]{\gfin(\cO(#1),\cO(#1))}
\nc{\gfinwox}[1]{\gfin(\Om(#1),\cO(#1))}
\nc{\gfinwlx}[1]{\gfin(\Om(#1),\Lambda(#1))}

\nc{\mc}{\mathcal}
\nc{\thusfar}{\my{--- Edited thus far ---}}
\nc{\lei}{\le^\oo}
\nc{\sqsubs}{\sqsubseteq^*}
\nc{\card}[1]{\left|#1\right|}
\nc{\medcard}[1]{\biggl|\,#1\,\biggr|}
\nc{\smallcard}[1]{|\,#1\,|}
\nc{\bds}{bidirectional $\roth$-scale}
\nc{\bfP}{\mathbf{P}}
\nc{\bfQ}{\mathbf{Q}}
\nc{\bbT}{\mathbb{T}}
\nc{\bbZ}{\mathbb{Z}}
\nc{\bbN}{\mathbb{N}}
\nc{\bbC}{\mathbb{C}}
\nc{\beq}{\begin{equation}}\nc{\eeq}{\end{equation}}
\nc{\mbq}{\mb{?}}
\nc{\mb}[1]{{\mbox{\textbf{#1}}}}
\nc{\nop}{$\times$}
\nc{\fbn}{\!\!\fbox{\!\nop\!}\!\!}
\nc{\yup}{\checkmark}
\nc{\forces}{\Vdash}
\nc{\name}[1]{\dot{#1}}
\nc{\tf}{\my{FINISHED THUS FAR}}
\nc{\FU}{Fr\'echet--Urysohn}
\nc{\gs}{$\gamma$~space}
\nc{\Gab}{\Gamma_{\mathrm{B}}}
\nc{\Omb}{\Omega_{\mathrm{B}}}
\nc{\Ga}{\Gamma}
\nc{\Om}{\Omega}
\nc{\smallbinom}[2]{\begin{psmallmatrix} #1\\ #2 \end{psmallmatrix}}
\nc{\bgamma}{\smallbinom{\Om}{\Ga}}

\nc{\productive}[2]{(#1,\allowbreak #2)^\x}
\nc{\prdct}[1]{#1^\x}
\nc{\Sel}{\mathsf{S}}
\nc{\sset}[2]{\{\,#1 : #2\,\}}
\nc{\smb}[1]{{\!\!\mb{#1}\!\!}}
\nc{\medset}[2]{{\biggl\{\,#1 : #2\,\biggr\}}}
\nc{\smallmedset}[2]{{\bigl\{\,#1 : #2\,\bigr\}}}
\nc{\set}[2]{{\left\{\,#1 : #2\,\right\}}}
\nc{\seq}[2]{{\la\, #1 : #2\,\ra}}
\nc{\eseq}[1]{#1_0, \allowbreak #1_1, \allowbreak\dotsc} 

\nc{\eseqint}[3]{#1_{#2}, \allowbreak\dotsc,\allowbreak #1_{#3}} 
\nc{\eseqstart}[2]{#1_{#2},\allowbreak #1_{#2+1},\dotsc } 
\nc{\eprod}[1]{#1_{1}\times \allowbreak#1_{2}\times\dotsb}

\nc{\shortprod}[1]{\prod_{n=1}^\infty{#1}_n}

\nc{\eprodint}[3]{#1_{#2}\times \allowbreak\dotsb\times\allowbreak #1_{#3}}

\nc{\seleseq}[1]{#1_1\in \mathcal{#1}_1, \allowbreak #1_2\in \mathcal{#1}_2, \allowbreak\dotsc}
\nc{\cube}{(\Cantor)^\bbN}
\nc{\Match}{\op{Match}}
\nc{\concat}[1]{\hat{\phantom{a}}\langle #1\rangle}
\nc{\poset}{\mathbb{P}}
\nc{\fn}[1]{{\op{Fn}(#1\times\w,2)}}
\nc{\linadd}{\op{linadd}}
\nc{\nonprod}{\non^\x}
\nc{\alephes}{{\aleph_0}}
\nc{\my}[1]{{\color{red}{#1}}}
\nc{\later}[1]{{\color{green} #1}}
\nc{\BTs}[1]{{\color{green} #1 (BT)}}
\nc{\Cp}{\op{C}_\mathrm{p}}
\nc{\Bp}{\op{B}_p}
\nc{\Pa}[8]{\bibitem{#1} {#2}, \emph{#3}, {#4} \textbf{#5} ({#6}), {#7}--{#8}.}
\nc{\tPa}[5]{\bibitem{#1} {#2}, \emph{#3}, {#4}, to appear.}
\nc{\sPa}[4]{\bibitem{#1} {#2}, \emph{#3}, {#4}, submitted.}
\nc{\Bc}[9]{\bibitem{#1} {#2}, \emph{#3}, in: \textbf{#4} (#5), #6 #7, #8--#9.}
\nc{\fD}{\mathfrak{D}}
\nc{\fX}{\mathfrak{X}}
\nc{\Onbd}{\Op_{\mathrm{nbd}}} 
\nc{\Omnb}{\Om_{\mathrm{nbd}}} 
\nc{\od}{\mathfrak{od}}
\nc{\Setting}[7]{\xymatrix@R=4pt@C=7pt{#1\ar@{-}[r]&#2\ar@{-}[r]&#3\\&#4\ar@{-}[u]\\
#5\ar@{-}[uu]\ar@{-}[r] & #6\ar@{-}[u]\ar@{-}[r] & #7\ar@{-}[uu]}}
\nc{\mx}[1]{\begin{matrix}#1\end{matrix}}
\nc{\plim}{p\txt{-}\lim}
\nc{\Bgp}{{\Z^\bbN}}
\nc{\Cgp}{{{\Z_2}^\bbN}}
\nc{\Cite}[1]{\textbf{[#1]}}
\nc{\Next}[1]{{#1^+}}
\nc{\cFin}{\mathrm{cF}}
\nc{\scsp}{\text{-scale space}}
\nc{\cfn}{\text{cofinal}\ }
\nc{\Con}{\text{Concentrated}}
\nc{\Lind}{\text{Lindel\"of}\,}
\nc{\con}{\text{-Concentrated}}
\nc{\lind}{\text{-Lindel\"of}\,}
\nc{\ctbl}{\text{countably }\allowbreak}
\nc{\Hur}{\text{Hurewicz}}
\nc{\intvl}[2]{{[#1(#2),\allowbreak #1(#2\!+\!1))}}
\nc{\Bdd}{\mathbf{B}}
\nc{\Dfin}{\mathfrak{D}_\mathrm{fin}}
\nc{\grbl}{{\mbox{\textit{\tiny gp}}}}
\nc{\bbP}{\mathbb{P}}
\nc{\bbM}{\mathbb{M}}
\nc{\bbQ}{\mathbb{Q}}
\nc{\bbH}{\mathbb{H}}
\nc{\BOfat}{\B_{\Om_{\mathrm{fat}}}}
\nc{\Bgood}{\B_{\mathrm{good}}}
\nc{\compactN}{\cl{\mathbb{N}}}
\nc{\blocks}[2]{\op{cl}_{#2}(#1)}
\nc{\blocksplus}[2]{\op{cl}^+_{#2}(#1)}
\nc{\arx}[1]{\texttt{http://arxiv.org/math/#1}}
\nc{\bq}{\begin{quote}}
\nc{\eq}{\end{quote}}
\nc{\cl}[1]{\overline{#1}}
\nc{\Cl}[2]{\mathrm{cl}_{#1}(#2)}
\nc{\CH}{the Continuum Hypothesis}
\nc{\MA}{Martin's Axiom}
\nc{\Bfat}{\B_\mathrm{fat}}
\nc{\inv}{^{-1}}
\nc{\Cantor}{{2^\w}}
\nc{\bP}{\mathbf{P}}
\nc{\bof}{\op{\fb}}
\nc{\dof}{\op{\fd}}
\nc{\bofF}{\bof(\cF)}
\nc{\sr}[3]{\underset{\mbox{#3}}{\mbox{#1}}}
\nc{\gp}{\binom{\Om}{\Ga}}
\nc{\gpsmall}{\mbox{$\gp$}}
\nc{\gig}{\gimel}
\nc{\gns}{\sone(\Om,\gig)}
\nc{\nsr}[2]{#1}
\nc{\Srg}{{\mathbb{S}}}
\nc{\Srgs}{{\mathbb{S}^*}}
\nc{\NN}{{\w^{\w}}}
\nc{\ZN}{{\Z^{\bbN}}}
\nc{\NNup}{{\w^{\uparrow\w}}}
\nc{\NNbarup}{{\overline{\w}}^{\uparrow\w}}
\nc{\NNupb}{{b^{\uparrow\bbN}}}
\nc{\Pof}{\op{P}}
\nc{\PN}{{\Pof(\w)}}
\nc{\rothx}[1]{{[#1]^{\mbox{\tiny $\infty$}}}}
\nc{\tx}{{\tilde{x}}}
\nc{\roth}{{[\w]^{\w}}}
\nc{\roths}{{[b]^{\mbox{\tiny $\infty$}}}} 
\nc{\Fin}{\mathrm{Fin}}
\nc{\ici}{[\bbN]^{ \infty, \infty}}
\nc{\Inc}{{\compactN^{\uparrow\bbN}}}
\nc{\powInc}[1]{{\big(\Inc\big)^{#1}}}
\nc{\powFin}[1]{{\big(\Fin\big)^{#1}}}
\nc{\powPN}[1]{{\big(\PN\big)^{#1}}}
\nc{\NcompactN}{{\compactN^\bbN}}
\nc{\Uarrow}{\smash{\big\uparrow}}
\nc{\LE}{\preccurlyeq}
\nc{\GE}{\succcurlyeq}
\nc{\op}{\operatorname}
\nc{\im}{\op{Im}}
\nc{\Span}{\op{span}}
\nc{\maxfin}{\op{maxfin}}
\nc{\ran}{\op{range}}
\nc{\iso}{\cong}
\nc{\Madd}{{\M}^\star}
\nc{\cI}{\mathcal{I}}
\nc{\cJ}{\mathcal{J}}
\nc{\scrA}{\mathscr{A}}
\nc{\scrB}{\mathscr{B}}
\nc{\scrC}{\mathscr{C}}
\nc{\scrD}{\mathscr{D}}
\nc{\scrF}{\mathscr{F}}
\nc{\scrK}{\mathscr{K}}
\nc{\A}{\D\forall}
\nc{\B}{\mathrm{B}}
\nc{\cB}{\mathcal{B}}
\nc{\cZ}{\mathcal{Z}}
\nc{\bB}{\mathbf{B}}
\nc{\BS}{\mathbf{B}(\mathcal{S})}
\nc{\BF}{\mathbf{B}(\mathcal{F})}
\nc{\BU}{\mathbf{B}(\mathcal{U})}
\nc{\cSp}{\mathcal{S}^+}
\nc{\cFp}{\mathcal{F}^+}
\nc{\cUp}{\mathcal{U}^+}

\nc{\BG}{\B_\Ga}
\nc{\BL}{\B_\Lambda}
\nc{\BT}{\B_\Tau}
\nc{\BTstar}{\B_{\Tau^*}}
\nc{\BO}{\B_\Om}
\nc{\DO}{\cD_\Om}
\nc{\KO}{\cK_\Om}
\nc{\CG}{C_\Ga}
\nc{\CL}{C_\Lambda}
\nc{\CT}{C_\Tau}
\nc{\CTstar}{C_{\Tau^*}}
\nc{\CO}{C_\Om}
\nc{\COgp}{C_{\Om^{\grbl}}}
\nc{\CLgp}{C_{\Lambda^{\grbl}}}
\nc{\BOgp}{\B_{\Om}^{\grbl}}
\nc{\BLgp}{\B_{\Lambda^{\grbl}}}
\nc{\sfC}{\mathsf{C}}
\nc{\sfD}{\mathsf{D}}
\nc{\bD}{\mathbf{D}}
\nc{\Tau}{\mathrm{T}}
\nc{\cA}{\mathcal{A}}
\nc{\cK}{\mathcal{K}}
\nc{\cD}{\mathcal{D}}
\nc{\cF}{\mathcal{F}}
\nc{\cS}{\mathcal{S}}
\nc{\cT}{\mathcal{T}}
\nc{\cG}{\mathcal{G}}
\nc{\cY}{\mathcal{Y}}
\nc{\J}{\mathcal{J}}
\nc{\cL}{\mathcal{L}}
\nc{\cM}{\mathcal{M}}
\nc{\cN}{\mathcal{N}}
\nc{\cH}{\mathcal{H}}

\nc{\Op}{\mathrm{O}}
\nc{\rmA}{\mathrm{A}}
\nc{\rmF}{\mathrm{F}}
\nc{\rmB}{\mathrm{B}}
\nc{\rmD}{\mathrm{D}}
\nc{\rmP}{\mathrm{P}}
\nc{\cC}{\mathcal{C}}
\nc{\cP}{\mathcal{P}}
\nc{\bbR}{\mathbb{R}}
\nc{\bbS}{\mathbb{S}}
\nc{\cU}{\mathcal{U}}
\nc{\cQ}{\mathcal{Q}}
\nc{\Un}{\bigcup}
\nc{\cV}{\mathcal{V}}
\nc{\cR}{\mathcal{R}}
\nc{\tcR}{\tilde{\mathcal{R}}}
\nc{\cW}{\mathcal{W}}
\nc{\Z}{{\mathbb Z}}
\nc{\Impl}{\Rightarrow}
\long\def\forget#1\forgotten{\marginpar{\textcolor{green}{Forgetting...}}}
\nc{\ft}{\mathfrak{t}}
\nc{\fb}{\mathfrak{b}}
\nc{\fc}{\mathfrak{c}}
\nc{\fd}{\mathfrak{d}}
\nc{\fg}{\mathfrak{g}}
\nc{\oo}{\infty}
\nc{\fr}{\mathfrak{r}}
\nc{\fk}{\mathfrak{k}}
\nc{\bidi}{\mathfrak{bidi}}
\nc{\fu}{\mathfrak{u}}
\nc{\fh}{\mathfrak{h}}
\nc{\fp}{\mathfrak{p}}
\nc{\fj}{\mathfrak{j}}
\nc{\fs}{\mathfrak{s}}
\nc{\w}{\omega}
\nc{\x}{\times}
\nc{\Iff}{\Leftrightarrow}

\nc{\nin}{\notin}
\nc{\cat}{\hat{\ }}
\nc{\sub}{\subseteq}
\nc{\spst}{\supseteq}
\nc{\sm}{\setminus}
\nc{\as}{\subseteq^*}
\nc{\les}{\le^*}
\nc{\leinf}{\le^{\infty}}
\nc{\leS}{\le_S}
\nc{\leF}{\le_F}
\nc{\leU}{\le_U}
\nc{\gew}{\geq^\textrm{w}}
\nc{\rest}{\restriction}

\nc{\la}{\langle}
\nc{\ra}{\rangle}
\nc{\E}{\exists}
\nc{\dom}{\op{dom}}
\nc{\cov}{\op{cov}}
\nc{\add}{\op{add}}
\nc{\addmen}{\add(\Men{})}
\nc{\cof}{\op{cof}}
\nc{\cf}{\op{cf}}
\nc{\non}{\op{non}}
\nc{\unif}{\op{non}}
\nc{\COV}{\op{COV}}
\nc{\ADD}{\op{ADD}}
\nc{\COF}{\op{COF}}
\nc{\NON}{\op{NON}}
\nc{\supp}{\op{supp}}
\nc{\impl}{\to}
\nc{\Lp}{\mathcal{L_\p}}
\nc{\Wlog}{without loss of generality}
\newtheorem{thm}{Theorem}[section]
\nc{\bthm}{\begin{thm}} \nc{\ethm}{\end{thm}}
\newtheorem{need}[thm]{Need}
\nc{\bneed}{\begin{need}\color{dg}} \nc{\eneed}{\end{need}}
\newtheorem{prop}[thm]{Proposition}
\nc{\bprp}{\begin{prop}} \nc{\eprp}{\end{prop}}
\newtheorem{fact}[thm]{Fact}
\nc{\bfct}{\begin{fact}} \nc{\efct}{\end{fact}}
\newtheorem{prob}[thm]{Problem}
\nc{\bprb}{\begin{prob}} \nc{\eprb}{\end{prob}}
\newtheorem{lem}[thm]{Lemma}
\nc{\blem}{\begin{lem}} \nc{\elem}{\end{lem}}
\newtheorem{app}[thm]{Application}
\nc{\bapp}{\begin{app}} \nc{\eapp}{\end{app}}
\newtheorem{claim}[thm]{Claim}
\nc{\bclm}{\begin{claim}} \nc{\eclm}{\end{claim}}
\newtheorem{cor}[thm]{Corollary}
\nc{\bcor}{\begin{cor}} \nc{\ecor}{\end{cor}}
\newtheorem{conj}[thm]{Conjecture}
\nc{\bcnj}{\begin{conj}} \nc{\ecnj}{\end{conj}}
\theoremstyle{definition}
\newtheorem{defn}[thm]{Definition}
\nc{\bdfn}{\begin{defn}} \nc{\edfn}{\end{defn}}
\newtheorem{obs}[thm]{Observation}
\nc{\bobs}{\begin{obs}} \nc{\eobs}{\end{obs}}
\theoremstyle{remark}
\newtheorem{rem}[thm]{Remark}
\nc{\brem}{\begin{rem}} \nc{\erem}{\end{rem}}
\newtheorem{cnv}[thm]{Convention}
\nc{\bcnv}{\begin{cnv}} \nc{\ecnv}{\end{cnv}}
\newtheorem{exam}[thm]{Example}
\nc{\bexm}{\begin{exam}} \nc{\eexm}{\end{exam}}
\nc{\bpf}{\begin{proof}} \nc{\epf}{\end{proof}
}
\nc{\be}{\begin{enumerate}}
\nc{\ee}{\end{enumerate}}
\nc{\bi}{\begin{itemize}}
\nc{\bimy}{\my{\begin{itemize}}
\nc{\eimy}{\end{itemize}}}
\nc{\itm}{\item}
\nc{\ei}{\end{itemize}}
\nc{\Subsection}[1]{\goodbreak\subsection*{#1}}

\nc{\sone}{\mathsf{S}_1}
\nc{\sfin}{\mathsf{S}_\mathrm{fin}}
\nc{\ufin}{\mathsf{U}_\mathrm{fin}}

\nc{\gone}{\mathsf{G}_1}    
\nc{\tgfin}{\tilde{\mathsf{G}}_\mathrm{fin}}
\nc{\gfin}{\mathsf{G}_\mathrm{fin}}

\nc{\men}[1]{\sfin(\Op(#1),\Op(#1))}
\nc{\sch}{\ufin(\cO,\Omega)}
\nc{\rothb}{\text{Rothberger}}
\nc{\pmen}{\sfin(\Omega,\Omega)}
\nc{\Rothb}{\sone(\Op,\Op)}

\nc{\prothb}{\sone(\Omega,\Omega)}
\nc{\tU}{{\tilde{U}}}
\nc{\tF}{{\tilde{F}}}
\nc{\tY}{{\tilde{Y}}}
\nc{\tX}{{\tsup[1]{X}}}
\nc{\dtX}{{\tsup[2]{X}}}
\nc{\dt}[1]{{\tsup[2]{#1}}}
\nc{\td}{{\tilde{d}}}
\nc{\tz}{{\tilde{z}}}
\nc{\cfd}{\cf(\fd)}

\nc{\msep}{\sfin(\cD,\cD)}
\nc{\rsep}{\sone(\cD,\cD)}
\nc{\cft}{\sfin(\Omega_{\mathbf{0}},\Omega_{\mathbf{0}})}
\nc{\scft}{\sone(\Omega_{\mathbf{0}},\Omega_{\mathbf{0}})}

\nc{\Umen}{U\text{-Menger}}
\nc{\hur}{\ufin(\cO,\Gamma)}
\nc{\tUmen}{\tU\text{-Menger}}

\nc{\Men}{\text{Menger}}
\nc{\Sch}{\text{Scheepers}}

\nc{\aspst}{\prescript{*}{}{\spst}\ }
\nc{\eqs}{=^*}
\nc{\ctblOm}{\Omega_{\mathrm{ctbl}}}
\nc{\GNga}{{\smallbinom{\Om}{\Ga}}}
\nc{\ctblga}{\smallbinom{\ctblOm}{\Ga}}

\nc{\nadd}{\cN_{\mathrm{add}}}
\nc{\ball}{\mathrm{B}}
\nc{\cOunif}{\cO^{\textrm{unif}}}

\nc{\sep}{
\vspace{2cm}
\noindent
\begin{minipage}{\textwidth}
	\textcolor{red}{\rule{\textwidth}{1pt}}
\end{minipage}
}
\nc{\FS}{\op{FS}}
\nc{\sums}{\op{SS}}
\nc{\SG}{\op{SG}}
\nc{\tSG}{\op{\widetilde{SG}}}
\nc{\G}{\op{G}}
\nc{\pBM}{\op{wgM}}
\nc{\FSG}{\op{FSG}}
\nc{\M}{\op{M}}
\nc{\gM}{\op{gM}}
\nc{\FP}{\op{FP}}
\nc{\nonNadd}{\non(\nadd)}
\nc{\borga}{\Ga_\mathrm{Bor}}
\nc{\pick}{x}
\nc{\gen}{y}
\nc{\nullzind}{\sone(\{\Op_n^{\mathsf{unif}}\}_{n\in\bbN},\Ga)}
\nc{\nullzindf}[1]{\sone(\{\Op_{#1}^{\mathsf{unif}}\}_{n\in\bbN},\Ga)}

\definecolor{dg}{RGB}{42,101,24}

\nc{\myb}[1]{\textcolor{blue}{#1}}
\nc{\mydg}[1]{{\color{dg}{#1}}}
\DeclareMathOperator{\eexists}{\exists}
\DeclareMathOperator{\fforall}{\forall}
\DeclareMathOperator{\fforallstar}{\forall^*}
\nc{\Exists}[1]{\bigl(\eexists #1\bigr)}
\nc{\Forall}[1]{\bigl(\fforall #1\bigr)}
\nc{\Forallstar}[1]{\bigl(\fforallstar #1\bigr)}
\nc{\End}[1]{\bigl(#1\bigr)}
\nc{\dmo}[2]{\DeclareMathOperator{#1}{#2}}
\dmo{\Asc}{Asc}
\nc{\plusmin}{\wedge}

\nc{\cBsub}{{\cB^{\mbox{\tiny $\sub$}}}}

\nc{\Alice}{{\textsc{Alice}{}}}
\nc{\Bob}{{\textsc{Bob}}}
\nc{\BM}{\op{BM}}
\nc{\Palpha}{{\bbS_\alpha}}
\nc{\Pbeta}{{\bbS_\beta}}
\nc{\Pwtwo}{\bbS_{\w_2}}
\nc{\restrict}{\upharpoonright}
\nc{\U}{\mathcal U}
\nc{\zrost}{\bar{\w}^{\uparrow\w}}

\newcommand{\spl}{\Split}

\title[Universally meager sets in the Miller 
model]{Universally meager sets in the Miller 
model and similar ones}

\subjclass[2020]{Primary: 03E35, 54D20, 54A35. Secondary: 
03E75.}
\keywords{Miller forcing,  Rothberger space, universally meager spaces, strong measure zero sets, Hurewicz space. }
\thanks{The research of the first and the third authors
was funded in whole by the Austrian Science Fund (FWF) [10.55776/I5930 and 10.55776/PAT5730424].
The research of the second author was funded by the National Science Center, Poland Weave-UNISONO call in the Weave programme
Project: Set-theoretic aspects of topological selections 2021/03/Y/ST1/00122
}

\author[V.~Haberl]{Valentin Haberl}
\address{Institut f\"ur Diskrete Mathematik und Geometrie, Technische Universit\"at Wien, Wiedner Hauptstrasse 8-10/104, 1040 Wien, Austria.}
\email{valentin.haberl.math@gmail.com}
\urladdr{https://www.tuwien.at/mg/valentin-haberl/}
\author[P.~Szewczak]{Piotr Szewczak}
\address{Institute of Mathematics, Faculty of Mathematics and Natural Science,
College of Sciences, Cardinal Stefan Wyszy\'nski University in Warsaw, W\'oycickiego 1/3,
01–938 Warsaw, Poland, and Institute of Mathematics, Faculty of Mathematics, Informatics and Mechanics, University of Warsaw, Banacha 2, 02-097, Warsaw, Poland}
\email{p.szewczak@wp.pl}
\urladdr{http://piotrszewczak.pl}
\author[L.~Zdomskyy]{Lyubomyr Zdomskyy}
\address{Institut f\"ur Diskrete Mathematik und Geometrie, Technische Universit\"at Wien, Wiedner Hauptstrasse 8-10/104, 1040 Wien, Austria.}
\email{lzdomsky@gmail.com}
\urladdr{https://dmg.tuwien.ac.at/zdomskyy/}

\begin{document}

\begin{abstract}
We work in the realm of sets of reals.
We prove that in the Miller model and in a model constructed by Goldstern--Judah--Shelah all universally meager sets have size at most $\w_1$.
Some relations between combinatorial covering properties in these models allow to obtain the same limitations for sizes of Rothberger spaces and Hurewicz spaces with no homeomorphic copy of the Cantor set inside.
It follows from our results that the existence of a strong measure zero set of size $\w_2$ does not imply the existence of a Rothberger space of size $\w_2$.
We also prove that in the Miller model all strong measure zero sets have size at most~$\w_1$.

\end{abstract}

\maketitle

\section{Introduction}

\bdfn[\cite{UM1}]
A subspace $X \sub 2^\w$ is \emph{universally meager}, if every Borel isomorphic image of $X$ is meager in $\Cantor$.
\edfn

The definition of universally meager subspaces was introduced by Zakrzewski~\cite{UM1} and shown by him to be equivalent~\cite{UM1} with an earlier studied by Grzegorek notion of \emph{absolutely of the first category} sets~\cite{grzegorek1980solution, grzegorek1984always, grzegorek1985always}.
Universally meager subspaces were also later  studied  by Zakrzewski who proved in \cite{UM2} several characterizations thereof which are crucial for us (see Lemma~\ref{lem:3.1}).
In some models of set theory there are limitations for sizes of universally meager subspaces.
This is the case, e.g., in the Cohen model (a forcing extension of a ground model of CH using a finite support iteration (equivalently, product) of the Cohen forcing, of length $\w_2$), where all universally meager subspaces have size at most $\w_1$, see the discussion after Problem~\ref{prob:6.5} for more details.
In our work we consider these limitations in the Miller model (a forcing extension of a ground model of GCH
using a countable support iteration forcing of length $\w_2$, introduced by Miller~\cite{17inHSZ1}) and in the model constructed by Goldstern--Judah--Shelah~\cite[the proof of Theorem~0.16]{GJS1993}.

\bthm\label{thm:UM}
In the Miller model and in the Goldstern--Judah--Shelah model, every universally meager subspace has size at most $\w_1$.
\ethm

The proof of Theorem~\ref{thm:UM} is based on the fact that
forcing notions involved satisfy the property $(\dagger)$ introduced in \cite{RepZd2019}.
It turns out that also the Cohen forcing shares the $(\dagger)$ property.
Our results mentioned above concern universally meager subspaces but the origin of our investigations is related to other properties of topological spaces and Theorem~\ref{thm:UM} helps to solve more problems.

A set $X\sub 2^\w$ has \emph{strong measure zero} (\emph{SMZ}, in short) if for every sequence
$\seq{\epsilon_n}{n\in\w}$ of positive reals there exists a sequence
$\seq{a_n}{n\in\w}$ of ``centers'' such that
the family $\sset{\B(a_n,\epsilon_n)}{n\in\w}$ covers $X$, where
$\B(a,\epsilon):=\sset{x\in\Cantor}{\rho(x,a)<\epsilon}$ 
and $\rho $ is any metric generating the standard topology on $2^\w$.
This notion was introduced by Borel who conjectured~\cite{Bor19}
that only countable sets have this property.
The Borel conjecture has been
consistently refuted by Sierpi\'{n}ski~\cite{Sie28} in 1928, and proved to be
consistent in 1976 by Laver~ \cite{Lav76}, and since then appears in different contexts  as well as serves as a blueprint for similar conjectures in set theory of 
the reals, see, e.g.,~\cite{GKWS14}.

Rothberger noticed~\cite{Rot38}, that being a strong measure set might not be a topological property and he modified it to the following topological covering property.
By \emph{space} we mean a topological space homeomorphic with a subspace of the Cantor cube $\Cantor$.
A space $X$ is \emph{Rothberger} if for any sequence $\seq{\cU_n}{n\in\w}$ of open covers of $X$, there are sets $U_n\in\U_n$ for $n\in\w$, such that the family $\sset{U_n}{n\in\w}$ covers $X$. 
It is easy to see that each Rothberger subset of $\Cantor$ has SMZ.

The first inspiration to this paper comes from our work~\cite{ConcMiller} on concentrated and $\gamma$-spaces.
A space $X$ is \emph{concentrated on a countable set} $D\sub X$, if the set $X\sm U$ is countable for any open set $U\sub X$ containing $D$.
A space $X$ is \emph{concentrated} if it is uncountable and concentrated on some of its countable subsets.
Bartoszy\'{n}ski--Halbeisen conjectured~\cite[Proposition~3.2]{3inHSZ1}
that there exists a concentrated space of size $\w_2$ in the Miller model.
This conjecture has been refuted~\cite{ConcMiller}, where we have also showed that any $\gamma$-space has size at most $\w_1$ in the Miller model (see \cite{11inHSZ1} for the definition of $\gamma$-spaces). 
Since the class of all Rothberger spaces of $\Cantor$ contains all concentrated and $\gamma$-spaces~\cite{11inHSZ1}, we have asked in our previous work~\cite{ConcMiller}, the following question.

\bprb[{\cite[Problem~4.2]{ConcMiller}}]\label{prb:RM}
Does every Rothberger space have size at most $\w_1$, in the Miller model?
\eprb

We get the following result which also solves Problem~\ref{prb:RM}.

\bthm\label{thm:R}
In the Miller model and in the Goldstern--Judah--Shelah model, every Rothberger space has size at most $\w_1$.
\ethm

Now we explain how Theorem~\ref{thm:UM} delivers the above results.
To this end we need an auxiliary notion of \emph{Hurewicz} spaces~\cite{Hur27}.
We shall not work directly with Hurewicz spaces,
but rather cite their known properties, hence we 
refer the reader to a work of Tsaban~\cite{Tsa_book} for more information of such spaces.
For definitions and properties of cardinal characteristics of the continuum used below, we refer the reader to a work of Blass~\cite{Bla10}.
There are four main steps to prove Theorem~\ref{thm:R}.

\begin{itemize}
\item The inequality $\fu<\fg$ holds in the Miller model and in the Goldstern--Judah--Shelah model.
\item In models of $\fu<\fg$, every Rothberger space is Hurewicz and \emph{totally imperfect}, i.e., does not contain a homeomorphic copy of $\Cantor$ inside~\cite{SemifilterZdo};
\item Every Hurewicz totally imperfect subspace of $2^\w$ is universally meager~\cite{UM1}. 
\item Apply Theorem~\ref{thm:UM}.
\end{itemize}

The following diagram presents relations between considered here properties for subspaces of $\Cantor$ (only one implication requests some additional assumptions beyond ZFC and it holds, e.g., when $\fu<\fg$).

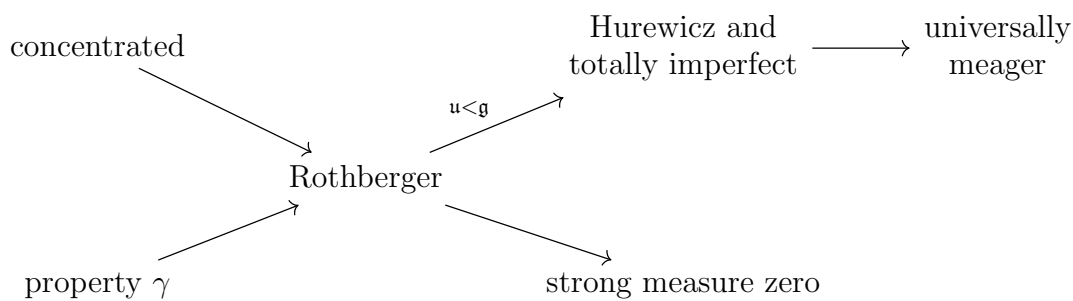
\begin{figure}[H]
\begin{tikzcd}[ampersand replacement=\&,column sep=1cm]
\text{concentrated}\arrow{rd} \& \& \begin{matrix}\text{Hurewicz and}\\
\text{totally imperfect}
\end{matrix}
\arrow{r} \& 
\begin{matrix}\text{universally}\\\text{meager}\end{matrix}\\
 \& \text{Rothberger}\arrow[ru, "\fu<\fg"] \arrow[rd] \& \&\\
\text{property }\gamma\arrow{ru} \& \& \text{strong measure zero} \& \\
\end{tikzcd}
\end{figure}

Some additional properties of the Goldstern--Judah--Shelah model allows to prove that, in this model, a space is Rothberger if and only if it is Hurewicz and totally imperfect (see Theorem~\ref{thm:gjs_R=H}).

In the above model, we have $\w_2=\fc=\fd>\fb$ and any union of less than $\fc$ many SMZ sets has SMZ (in particular each set of size less than $\fc$ has SMZ).
If $\fc=\fd$ and each set of size less than $\fc$ has SMZ, then there is a SMZ set of size $\fc$, which is the case in the Goldstern--Judah--Shelah model.
Combining this with Theorem~\ref{thm:R}, we show that the existence of a SMZ set of size $\w_2$ does not imply the existence of a Rothberger space of size $\w_2$.
This could be compared with a result of Miller~\cite[Proposition~8]{Mil05} stating that such an implication holds in the case of $\w_1$.

\bthm\label{main_gjs} 
In the Goldstern--Judah--Shelah model,  there exists a SMZ set of size $\w_2 = \mathfrak{c}$, but  every Rothberger space has size at most $\w_1$. 
\ethm

Theorem~\ref{main_gjs} can be compared with Judah's result
\cite[Theorem~1.3]{Jud88} which gives that the existence of
SMZ sets of size $\hot c$, which can be arranged to be $\w_2$ there,
does not imply the existence of some specific Rothberger spaces of
size $\hot c$, namely generalized Luzin sets. However,
it is easy to see that the ground model reals 
form a Rothberger space of size $\hot c$ in the model constructed in the proof of 
\cite[Theorem~1.3]{Jud88}. Thus, 
the Goldstern--Judah--Shelah model from \cite[Theorem~0.16]{GJS1993}
cannot be replaced with Judah's  model from \cite[Theorem~1.3]{Jud88}
in Theorem~\ref{main_gjs}.

The \emph{weak Borel conjecture} is the statement that there are no SMZ sets of size $\mathfrak{c}$. In \cite[Corollary~3.6]{GJS1993}, it was shown that models obtained by using a countable support iteration of proper strongly $\w^\w$-bounding forcing notions that preserve CH in every intermediate model satisfy the weak Borel conjecture.
In \cite{BarShel03} Bartoszy\'{n}ski and Shelah showed that the weak dual Borel conjecture holds in the Laver model and in the Miller model. They were interested in finding a model, where both the Borel conjecture and the dual Borel conjecture hold, which was finally found in \cite{GKSW14}. However, it is still an open question if the Laver model is already a witness for both conjectures holding simultaneously. In \cite{Woh20} it was pointed out that it is an open problem if the Miller model satisfies the weak Borel conjecture. We address this problem in Section~5, delivering the following answer.

\bthm\label{main_smz}
In the Miller model, every SMZ set has size at most $\w_1$, i.e., the weak Borel conjecture holds.
\ethm

A subspace $X\sub\Cantor$ is \emph{perfectly meager} if for any perfect set $P\sub\Cantor$, the set $X\cap P$ is meager in the relative topology on $P$.
The proof of Theorem~\ref{main_smz} uses an idea from~\cite{JudShe94} allowing to show that, in the Miller model, each SMZ set is perfectly meager.
By the result of Bartoszy\'{n}ski~\cite[Theorems 9 and 11]{perfectlymeager2002}, in the Miller model, every perfectly meager subspace is universally meager.
In the light of the above mentioned results, Theorem~\ref{main_smz} is a consequence of Theorem~\ref{thm:UM}.


\section{Property $(\dagger)$}

The following notion was introduced in \cite{RepZd2019}.
All posets we consider are separative.

\bdfn
A forcing $\mathbb{P}$ has property $(\dagger)$
if for every countable elementary submodel $M\ni \IP$ of $H(\theta)$ for sufficiently large $\theta$, condition $p \in \IP\cap  M$, and functions $\varphi_i \colon \IP\cap M \to \IP\cap M$ for all $i\in\w$  such that $\varphi_i(p) \leq p$
for all $p \in \IP \cap M$ and $i \in\w$, there exists an $(M, P)$-generic condition $q\leq p$
 forcing
\begin{equation} \label{dag_orig}
\mbox{``}\Gamma_{\IP}\cap \{\varphi_i(p) : p \in\IP\cap M\}\ \  \mbox{is infinite for all} \ \ i\in\w\mbox{''} ,
\end{equation}
where $\Gamma_\IP$ is the canonical $\IP$-name for the $\IP$-generic filter.

Using notations from above, any function $\varphi\colon \bbP\cap M\to \bbP\cap M$ such that $\varphi(p)\leq p$ for all $p\in \bbP\cap M$ will be called \emph{regressive}.
\edfn

For our purposes in this work we shall need an equivalent 
formulation of $(\dagger)$ provided by the following straightforward result.

\blem \label{l_1.2}
The property
$(\dagger)$ of a forcing $\bbP$ is equivalent to its  version obtained 
by replacing statement~(\ref{dag_orig}) with the following formally weaker one:
\[
\mbox{``}\,\Gamma_{\IP}\cap \{\varphi_i(p) : p \in\IP\cap M\}\neq\emptyset \  \mbox{for all} \ \ i\in\w\mbox{''}.  
\]
\elem

\bpf
It suffices to prove that, if $\IP$ satisfies the condition in our Lemma, then it fulfills $(\dagger)$.
Using the separativity of $\IP$, for every $i\in\w$ there is a sequence $\seq{\varphi^i_j}{j\in\w}$ of regressive functions
from $M\cap \IP$ to $M\cap \IP$
such that 
\begin{gather*}\varphi_i(p)>\varphi^i_0(p)\geq \varphi_i(\varphi^i_0(p))>\varphi^i_1(p)\geq \varphi_i(\varphi^i_1(p))>\\
>\varphi^i_2(p) \geq \varphi_i(\varphi^i_2(p))> \cdots >\varphi^i_j(p)\geq \varphi_i(\varphi^i_j(p))>\varphi^i_{j+1}(p)\geq\cdots
\end{gather*}
for all $p\in\IP\cap M$ and $i,j\in\w$.
Hence, we obtain that for every $p\in\IP\cap M$ and $i,j\in\w$ the set
\[
Q^i_j(p):=\sset{r\in \varphi_i[M\cap\bbP]}{\varphi^i_j(p)\leq r}
\]
has size at least $j+1$.
Then there exists an $(M, \IP)$-generic condition $q\leq p$
forcing
\[
\mbox{``}\,\Gamma_{\IP}\cap \{\varphi^i_j(p) : p \in\IP\cap M\}\neq\emptyset \  \mbox{ for all} \ \ i,j\in\w\mbox{''}.
\]
Let $G$ be a $\IP$-generic filter over $V$ with $q\in G$. 
In $V[G]$, fix $p^i_j\in\IP\cap M$ such that $\varphi^i_j(p^i_j)\in G$ for all $i,j$.
For each $i$, by the above, the union $\bigcup_{j\in\w}Q^i_j(p^i_j)$ is an infinite subset of $\varphi_i[M\cap\bbP]\cap G$.
\epf

\brem
The range of a function $s$ we denote by $\rng(s)$.
For a finite sequence $t$ we shall  denote by $t(\text{end}) := t(|t|-1)$, i.e.,  the value of $t$ at its last element. 
For a tuple $\la a,b\ra$ we denote with $\mathrm{dom}(\la a, b\ra) = a$ and $\rng(\la a,b \ra) = b$. For an iterated forcing construction
$\la \IP_\alpha,\dot{\IQ}_\alpha:\alpha < \delta\ra$
 we use shorter notation  $\Gamma_\alpha$ and $\Gamma_{[\alpha,\beta)}$ for   
 $\Gamma_{\mathbb{P}_\alpha}$ and $\Gamma_{\mathbb{P}_{[\alpha,\beta)}}$, respectively. \hfill $\Box$
\erem

\blem
Let $\mathbb{P}$ be a forcing with property $(\dagger)$ and let $\dot{\mathbb{Q}}$ be a name for a forcing such that 
$$\mathbbm 1_\mathbb{P} \forces "\dot{\mathbb{Q}} \text{ satisfies } (\dagger)``.$$
Then $\mathbb{P} * \dot{\mathbb{Q}}$ satisfies $(\dagger)$.
\elem
\bpf
Fix a countable elementary submodel $M \ni \mathbb{P}*\dot{\mathbb{Q}}$ of $H(\theta)$ for sufficiently large $\theta$. 
Let $\seq{\varphi_n}{n \in \w}$ be a sequence of regressive functions from $M \cap (\mathbb{P}*\dot{\mathbb{Q}})$ to $M \cap (\mathbb{P}*\dot{\mathbb{Q}})$.
Suppose that $\la p_0, \dot{q}_0 \ra \in M \cap (\mathbb{P}*\dot{\mathbb{Q}})$.
For every $n \in \w$ and $\dot{q} \in M $ with $\mathbbm{1}_\mathbb{P} \forces "\dot{q} \in \dot{\mathbb{Q}}``$ define $\psi_{n,\dot{q}}\colon M \cap \mathbb{P} \rightarrow M \cap \mathbb{P}$ by setting $\psi_{n,\dot{q}}(p) := \mathrm{dom}(\varphi_n(\la p,\dot{q} \ra))$ for all $p \in M \cap \mathbb{P}$.
We clearly have $\psi_{n,\dot{q}}(p) \leq p$ for all $p$, and  there are only countably many $\psi_{n,\dot{q}}$'s.
Hence, by property $(\dagger)$, there is an $(M,\mathbb{P})$-generic condition $r \leq p_0$ such that for every $n \in \w$ and $\dot{q}$ as above we have 
\[
r \forces "\exists p_{n,\dot{q}} \in \mathbb{P} \cap M \enskip (\psi_{n,\dot{q}}(p_{n,\dot{q}}) \in \Gamma_\mathbb{P})``.
\]

Let $G$ be a $\mathbb{P}$-generic filter over $V$ with $r\in G$.
In $V[G]$, for every $n$ and $\dot{q}$, pick $p_{n,\dot{q}}$ with $\psi_{n,\dot{q}}(p_{n,\dot{q}}) = \mathrm{dom}(\varphi_n(\la p_{n,\dot{q}},\dot{q} \ra) \in G$.
Define regressive functions $\eta_n\colon M[G] \cap \dot{\mathbb{Q}}^G \rightarrow M[G] \cap \dot{\mathbb{Q}}^G$ for all $n \in \w$ as follows:
For $b \in M[G] \cap \dot{\mathbb{Q}}^G$ pick $\dot{q}_b \in M$ such that $\mathbbm{1}_\mathbb{P} \forces "\dot{q}_b \in \dot{\mathbb{Q}}``$ and $\dot{q}_b^G = b$. Set 
$$\eta_n(b) = \rng(\varphi_n(\la p_{n,\dot{q}_b},\dot{q}_b \ra))^G.$$
By property $(\dagger)$, there is an $(M[G],\dot{\mathbb{Q}}^G)$-generic condition $w \leq \dot{q_0}^G$ such that 
\[
w \forces "\exists b_n \in M[G] \cap \dot{\mathbb{Q}}^G \enskip (\eta_n(b_n) \in \Gamma_{\dot{\mathbb{Q}}^G})``
\]
for all $n \in \w$.
Let $H$ be a $\dot{\mathbb{Q}}^G$-generic filter over $V[G]$ with $w\in H$. In $V[G*H]$, for each $n \in \w$ there is $b_n \in M[G]\cap \dot{\mathbb{Q}}^G$ such that $\rng(\varphi_n(\la p_{n,\dot{q}_{b_n}}, \dot{q}_{b_n} \ra))^G \in H$.
By the definition of $p_{n,\dot{q}_{b_n}}$, we also have $\mathrm{dom}(\varphi_n(p_{n,\dot{q}_{b_n}}, \dot{q}_{b_n} \ra) \in G$. 

By the Maximality Principle, there are a $\mathbb{P}$-name $\dot{w}$ for $w$ and for every $n \in \w$ a $(\mathbb{P}*\dot{\mathbb{Q}})$-name $\dot{u}_n\in M$  such that 
$\mathbbm{1}_{\mathbb{P}*\dot{\IQ}}\forces"\dot{u}_n\in \mathbb{P}*\dot{\IQ}``$ and
$$\la r, \dot{w} \ra \forces "\mathrm{dom}(\varphi_n(\dot{u}_n)) \in \Gamma_\mathbb{P} \text{ and } \rng(\varphi_n(\dot{u}_n)) \in \Gamma_{\dot{\mathbb{Q}}}``,$$
which implies 
$$\la r, \dot{w} \ra \forces "\varphi_n(\dot{u}_n) \in \Gamma_{\mathbb{P}*\dot{\mathbb{Q}}}``.$$
Note that $\la r, \dot{w} \ra$ is an $(M, \mathbb{P}*\dot{\mathbb{Q}})$-generic 
condition below $\la p_0, \dot{q}_0 \ra$. Thus, $\mathbb{P}*\dot{\mathbb{Q}}$ satisfies $(\dagger)$. 
\epf

\bthm \label{preservation_dagger}
Let $\la \mathbb{P}_\alpha, \dot{\mathbb{Q}}_\alpha : \alpha < \delta \ra$ be a countable support iteration of forcings satisfying property $(\dagger)$ with the limit $\mathbb{P}_\delta$.
Then $\mathbb{P}_\delta$ satisfies $(\dagger)$. 
\ethm
\bpf
We shall assume that $(\dagger)$ is preserved by any 
iterations of length $<\delta$ with countable supports. 
Fix a countable elementary submodel $M \ni \delta, \mathbb{P_\delta}, \la \mathbb{P}_\alpha, \dot{\mathbb{Q}}_\alpha : \alpha < \delta \ra$ of $H(\theta)$ for sufficiently large $\theta$ and let $\la \alpha_n: n \in \w \ra \in (M \cap \mathrm{Ord})^\w$ be a strictly increasing cofinal sequence in $M \cap \delta$. Let $\la D_n: n \in \w\ra$ be an enumeration of all open dense subsets of $\mathbb{P}_\delta$, which lie in $M$. Suppose $\{\varphi_n: n \in \w\}$ is a collection of functions like in the definition of $(\dagger)$ and pick some $p_0 \in \mathbb{P_\delta}\cap M$. Without loss of generality we can assume that $\varphi_n[M \cap \mathbb{P}_\delta] \subseteq D_{n}$ for all $n \in \w$.

By induction on $n$ we pick $q_n\in\IP_{\alpha_n}$ and a name
$\dot{p}_n\in V^{\IP_{\alpha_n}}$ such that
\begin{itemize}
\item[$(i)$] $q_n\in\IP_{\alpha_n}$ is $(M,\IP_{\alpha_n})$-generic,
$q_n\uhr\alpha_{n-1}=q_{n-1}$ for all $n\geq 1$;
\item[$(ii)$] $\dot{p}_0=\check{p}$ and $q_n\forces_{\IP_{\alpha_n}}$ 
    \begin{itemize}
    \item[$(a)$] $ \dot{p}_n\in\IP_\delta\cap M$,
    \item[$(b)$] $\varphi_n(\dot{p}_n)\uhr\alpha_n\in\Gamma_{\alpha_n}$, 
    \item[$(c)$] $\dot{p}_n\uhr[\alpha_{n-1},\delta)= \varphi_{n-1}(\dot{p}_{n-1})\uhr[\alpha_{n-1},\delta)$ if\footnote{Here we get by the construcion a bit more, namely that   $\dot{p}_{n} \restriction
[\alpha_{n-1}, \delta)$ and $\varphi_{n-1}(\dot{p}_{n-1})\uhr[\alpha_{n-1},
\delta)$ are forced by $q_n$ to be equal ground-model objects.}
    $n\geq 1$,  and
    \item[$(d)$] $\forall u\in\IP_{[\alpha_n,\delta)}\cap M\:\forall\,i\in\w\:\exists\,s=s^n(u,i)\in\IP_{\alpha_n}\cap M$\\
    $\big(\varphi_i(s\bigvid u)\uhr\alpha_n \in\Gamma_{\alpha_n}\big)$.
    \end{itemize}
\end{itemize}
 First we construct $q_0$ such that $(i)$ and $(ii)$ are satisfied for $n=0$.
For every $r\in \IP_{\alpha_0}\cap M$ and $u\in \IP_{[\alpha_0,\delta)}\cap M$
set $\varphi^0_{i,u}(r)=\varphi_i(r\bigvid u)\uhr\alpha_0$ and note that
$\varphi^0_{i,u} \colon \IP_{\alpha_0}\cap M\to \IP_{\alpha_0}\cap M$ is a regressive function. Since $\IP_{\alpha_0}$ satisfies $(\dagger)$, there exists an $(M,\IP_{\alpha_0})$-generic condition $q_0\leq\varphi_0(p_0)\uhr\alpha_0$
such that 
\begin{equation} \label{eq_q_0}
q_0\forces "\forall i\in\w\: \forall u\in \IP_{[\alpha_0,\delta)}\cap M\enskip \big( \Gamma_{\alpha_0}\cap \varphi^0_{i,u}[M\cap\IP_{\alpha_0}]\neq\emptyset\big)``.
\end{equation}
Thus, $q_0$ and $\dot{p}_0$ satisfy $(i)$ and $(ii)$,
some parts of these conditions being of course vacuous for $n=0$.

Suppose that for some $n\in\w$ and all $k\leq n$ we have constructed $\dot{p}_k,q_k$ satisfying 
$(i)$ and $(ii)$. Let us fix a $\IP_{\alpha_n}$-generic filter $G$ over $V$ containing 
$q_n$ and work for a while in $V[G]$. From $(ii)(d)$ we can find
for every  $r\in\IP_{[\alpha_n,\alpha_{n+1})}\cap M$, $u\in\IP_{[\alpha_{n+1},\delta)}\cap M$ and $i\in\w$ a condition
$s=s^n(r,  u,i)\in \IP_{\alpha_n}$ such that
\begin{equation} \label{regr_def}
    \varphi_i(s\bigvid r\bigvid u)\uhr\alpha_n \in G.
\end{equation}
More precisely, we simply denote by  $s^n(r,u,i)$ (or just $s$ if $n,r,u,i$ are clear from the context) the condition
$s^n(r\bigvid u,i)$ provided by $(ii)(d)$.
It follows  that 
$$p_{n+1}:=s^n\big(\varphi_n(p_n)\uhr\big[\alpha_n,\alpha_{n+1}\big), \:\varphi_n(p_n)\uhr\big[\alpha_{n+1},\delta\big),\: n+1\big)\bigvid \varphi_n(p_n)\uhr[\alpha_n,\delta) $$
satisfies $(ii)(c)$ for $n+1$. 
 
Let us note that (\ref{regr_def}) 
 allows us to define a regressive map $\varphi^{n+1}_{i,u}$
from  $\IP^G_{[\alpha_n,\alpha_{n+1})}\cap M $ to\footnote{Note that 
the underlying set of $\IP^G_{[\alpha_n,\alpha_{n+1})}$ is in $V$ and does not depend on $G$, see, e.g., the discussion at the beginning of p. 23 in \cite{Bau83},  while the preorder relation depends on $G$. Thus, 
$\IP^G_{[\alpha_n,\alpha_{n+1})}\cap M =
\IP^G_{[\alpha_n,\alpha_{n+1})}\cap M[G].$} 
$\IP^G_{[\alpha_n,\alpha_{n+1})}\cap M$
for all $i,u$ as above in the following way:
\begin{equation}
  \varphi^{n+1}_{i,u}(r)=\varphi_i\big(s^n(r,u,i)\bigvid r\bigvid u\big)\uhr [\alpha_n,\alpha_{n+1}).  
\end{equation}
Since $\IP^G_{[\alpha_n,\alpha_{n+1})}$ satisfies $(\dagger)$, there exists
an $(M[G],\IP^G_{[\alpha_n,\alpha_{n+1})})$-generic condition $w_n$  such that 

\[
w_n\leq \varphi_n(p_n)\uhr [\alpha_n,\alpha_{n+1})=p_{n+1}\uhr [\alpha_n,\alpha_{n+1})
\]
and
\begin{equation} \label{ext_gen}
w_n\forces_{V[G]} 
 "\forall i\in\w\: \forall u\in \IP^G_{[\alpha_{n+1},\delta)}\cap M\: \big( \Gamma^G_{[\alpha_n,\alpha_{n+1})}\cap \varphi^{n+1}_{i,u}[M\cap\IP^G_{[\alpha_n,\alpha_{n+1})}]\neq\emptyset\big)``.
\end{equation}
Thus, if $H\ni w_n$ is a $\IP^G_{[\alpha_n,\alpha_{n+1})}$-generic filter over
$V[G]$, then in $V[G*H]$ we get from (\ref{ext_gen}) that for every $u\in \IP^{G*H}_{[\alpha_{n+1},\delta)}\cap M$ and $i\in\w$ there exists $r(u,i)\in \IP^{G*H}_{[\alpha_n,\alpha_{n+1})}\cap M$ such that
$\varphi^{n+1}_{i,u}\big(r(u,i)\big)\in H$, i.e., 
\begin{equation} \label{on_t_u_i}
    \varphi_i\big(s^n\big(r(u,i),u,i\big)\bigvid r(u,i)\bigvid u\big)\uhr [\alpha_n,\alpha_{n+1})\in H.  
\end{equation}
Now we work again in $V$.
Using the Maximality Principle several times we can find
the following objects: 
\begin{itemize}
\item A $\IP_{\alpha_n}$-name $\dot{w}_n$ which is forced by $q_n$
to be $(M[\Gamma_{\alpha_n}],\IP^{\Gamma_{\alpha_n}}_{[\alpha_n,\alpha_{n+1})})$-generic;
\item For every $u\in\IP_{[\alpha_{n+1},\delta)}\cap M$ and $i\in\w$ a 
$\IP_{\alpha_{n+1}}$-name $r(u,i)$ for a condition in $\IP_{[\alpha_n,\alpha_{n+1})}$, such that $q_{n+1}:=q_n\bigvid\dot{w}_n\in \IP_{\alpha_{n+1}}$
forces 
\begin{eqnarray} \label{on_t_u_i_forced}
    \varphi_i\big(s^n\big(r(u,i),u,i\big)\bigvid r(u,i)\bigvid u\big)\uhr \alpha_n\in \Gamma_{\alpha_n} \ \wedge\\
     \varphi_i\big(s^n\big(r(u,i),u,i\big)\bigvid r(u,i)\bigvid u\big)\uhr [\alpha_n,\alpha_{n+1})\in \Gamma_{[\alpha_n,\alpha_{n+1})}, \label{last_on_q}
\end{eqnarray}
\end{itemize}
\noindent \ (\ref{on_t_u_i_forced}) being a consequence of 
(\ref{regr_def}).
From (\ref{on_t_u_i_forced}) and (\ref{last_on_q})
we conclude that $q_{n+1}$ defined above along with the 
$\IP_{\alpha_{n+1}}$-names
$$ s^{n+1}(u,i)= s^n\big(r(u,i),u,i\big)\bigvid r(u,i), $$
satisfy $(i)$ and $(ii)$ for $n+1$, which completes our recursive construction of 
objects satisfying $(i)$ and $(ii)$.

We claim that  $q:=\bigcup_{n\in\w}q_n$ is an $(M,\IP_\delta)$-generic condition 
witnessing $(\dagger)$. Since the range of $\varphi_n$ is a subset of $D_n$, it suffices to show that $q$ forces $\varphi_n(\dot{p}_n)\in \Gamma_\delta$
for all $n\in\w$. By induction we show that for all $k \in \w$
\begin{equation} \label{dagger_3_proof}
q\forces " \forall n \leq k \enskip \varphi_n(\dot{p}_n)\uhr\alpha_k \in \Gamma_{\alpha_k}``.
\end{equation}
For $k = 0$ this is clear by $(ii)(b)$. Assume that (\ref{dagger_3_proof}) is true for $k$. Pick $n \leq k+1$. The case $n = k+1$ follows from $(ii)(b)$. Now suppose $n \leq k$. We have $q \forces "\varphi_n(\dot{p}_n) \restriction \alpha_k \in \Gamma_{\alpha_k}``$.  Furthermore, we have from $(ii)(b)$ and the induction hypothesis that 
$$q \forces "\forall j \in [n,k+1] \enskip (\varphi_j(\dot{p}_j) \restriction \alpha_k \in \Gamma_{\alpha_k})``.$$
Hence, by using $(ii)(c)$ for $k+1,k,k-1,\ldots,n+1$ instead of $n$, we obtain
\begin{eqnarray*}
q\forces "\varphi_{k+1}(\dot{p}_{k+1})\uhr[\alpha_k,\alpha_{k+1}) \leq \dot{p}_{k+1}\uhr[\alpha_k,\alpha_{k+1})= \varphi_k(\dot{p}_{k})\uhr[\alpha_k,\alpha_{k+1})\leq  
 \\
\leq \dot{p}_{k}\uhr[\alpha_k,\alpha_{k+1})=\cdots = \varphi_{n+1}(\dot{p}_{n+1})\uhr[\alpha_k,\alpha_{k+1})\leq \\
\leq \dot{p}_{n+1}\uhr[\alpha_k,\alpha_{k+1})= \varphi_n(\dot{p}_{n})\uhr[\alpha_k,\alpha_{k+1})``.
\end{eqnarray*}
Since $q\forces "\varphi_{k+1}(\dot{p}_{k+1})\uhr[\alpha_k,\alpha_{k+1})\in\Gamma_{[\alpha_k,\alpha_{k+1})}``$, we obtain 
$q \forces "\varphi_n(\dot{p}_{n})\uhr[\alpha_k,\alpha_{k+1}) \in \Gamma_{[\alpha_k,\alpha_{k+1})}``$. Thus, 
$$q \forces "\varphi_n(\dot{p}_{n}) \restriction \alpha_{k+1} \in \Gamma_{\alpha_{k+1}}``,$$
which completes our proof.
\epf

\section{Property $(\dagger)$ and universally meager subspaces}

Zakrzewski proved in \cite{UM1,UM2} among others many characterizations of universally meager subspaces. 
It will be convenient for us to use the following property characterizing universally meager subspaces
(see \cite[Theorem~1.1]{UM2}) as a definition.

\blem[{Zakrzewski \cite[Theorem~1.1]{UM2}}] \label{lem:3.1}
A set $X \subseteq 2^\w$ is universally meager if and only if for every Polish space $Y$ and continuous nowhere constant map $f\colon Y \rightarrow 2^\w$ the preimage $f^{-1}[X]$ is meager in $Y$.
\elem

The following Lemma is rather standard.

\blem \label{mod_groundmodel}
Suppose that CH holds in the ground model $V$  and
let $\la\mathbb{P}_\alpha, \dot{\mathbb{Q}}_\alpha : \alpha < \w_2 \ra$ be a countable support iteration of proper forcings such that 
$\mathbbm{1}_{\mathbb{P}_\alpha} \forces"|\dot{\mathbb{Q}}_\alpha| \leq \w_1``$
for all $\alpha < \w_2$. 
Let also $G_{\w_2}$ be a $\mathbb{P}_{\w_2}$-generic filter over $V$ and  $X\in V[G]$ be a universally meager subspace of $2^\w $.   Then there exists an $\w_1$-CLUB $C$ of ordinals $\alpha < \w_2$ with the following properties:
\begin{itemize}
\item $X \cap V[G_\alpha] \in V[G_\alpha]$;
\item If $Y$ is a Polish space, $f:Y \rightarrow 2^\w$ is a continuous nowhere constant map, both coded in $V[G_\alpha]$, then $f^{-1}[X]$ is contained in a meager subset of $Y$ coded in $V[G_\alpha]$;
\item If $Y$ is a Polish space, $f:Y\rightarrow 2^\w$ is a continuous map, both coded in $V[G_\alpha]$, and $R$ is a $G_\delta$-subset of $Y$ coded in $V[G_\alpha]$ such that $f[R] \cap V[G_\alpha] \cap X = \emptyset$, then $f[R] \cap X = \emptyset$.
\end{itemize}
As a consequence of the first two items, we get that $X \cap V[G_\alpha]$ is universally meager in $V[G_\alpha]$ for all $\alpha \in C$.
\elem

The proof is standard and is left as an exercise to the reader.

\bthm \label{main_result}
Suppose that  $\la\mathbb{P}_\alpha, \dot{\mathbb{Q}}_\alpha : \alpha < \w_2 \ra$ is such as in Lemma~\ref{mod_groundmodel} and
$\mathbbm{1}_{\mathbb{P}_\alpha} \forces "\dot{\mathbb{Q}}_\alpha$  satisfies  $(\dagger)``$
 for all $\alpha < \w_2$.
Then in $V^{\IP_{\w_2}}$ all universally meager subspaces  have size at most $\w_1$.
\ethm
\bpf
Let $\IP=\IP_{\w_2}$ and $G$ be a $\bbP$-generic filter over $V$.
Let $X\sub\Cantor$ be a universally meager subspace in $V[G]$ and $\name{X}$ be a name for the set $X$, i.e., $\name{X}^G=X$.
Let us pick $\alpha < \w_2$ in the $\w_1$-CLUB provided by Lemma \ref{mod_groundmodel}.  Suppose, towards a contradiction, that $\tau$ is a name for a real such that $r \forces "\tau \in \dot{X} \setminus V[G_\alpha]``$ for some condition $r\in G$. 
In what follows we assume without loss of generality that
$V=V[G_\alpha]$ and identify $G_{[\alpha,\w_2)}$ with $G$.

Until the opposite is stated, we work in $V$. 
For each $n\in \w$, let $D_n \sub \bbP$ be a dense open set below $r$ deciding $\tau \restriction n$.
Suppose $M$ is a countable elementary submodel of $H(\theta)$ for large enough $\theta$ such that $r, \mathbb{P}, \tau, \dot{X}, D_n \in M$. 
Equip $D_n \cap M$ with the discrete topology and consider a subspace $Y \sub \Pi_{n \in \w} (D_n \cap M)$ such that
\[
Y:=\smallmedset{\seq{p_n}{n \in \w} \in \Pi_{n \in \w} (D_n \cap M)}{p_{n+1} \leq p_n\text{ for all }n \in \w}.
\]
Since $Y$ is a nowhere locally compact, closed subspace of the zero-dimensional Polish space $\Pi_{n \in \w} (D_n \cap M)$, the space $Y$ is homeomorphic with $\NN$. 

For $y\in Y$ and $n\in\w$, the condition $y(n)$ decides $\tau\restriction n$, and thus there is $t_n\in 2^{<\w}$ such that
\[
y(n) \forces "\tau \restriction n = t_n``.
\]
Then a function $f\colon Y \to 2^\w$ defined by $f(y):= \bigcup_{n \in \w} t_n$ is a continuous function.  Since $\tau$ is forced to be not in $V$, the function $f$ is nowhere constant.
The set $X \cap V$ is universally meager in $V$, and thus the set 
\[
f^{-1}[(X \cap V) \cap f[Y]]
\] is meager in $Y \cap V$.
Let $\seq{O_n}{n \in \w}$ be a decreasing sequence of dense open sets in $Y$ such that 
\[
\bigcap_{n \in \w} O_n \cap f^{-1}[(X \cap V) \cap f[Y]] = \emptyset.
\]

If $t=y'\restriction k$ for some $y'\in Y$ and $k\in\w$, then $[t]:=\sset{y\in Y}{y\restriction k =t}$ is a basic open set in $Y$.
Let
\[
C_n := \sset{t \in \bigcup_{k \in \w} Y \restriction k}{[t] \sub O_n}
\]
for all $n\in\w$.
Then the set $C := \bigcup_{n \in \w} C_n$ is countable. 

Now we shall define a family of regressive functions to apply property~$(\dagger)$.

Define $\varphi^0 \colon M \cap \mathbb P \rightarrow M \cap \mathbb P$ as follows:
Fix $p \in M \cap \mathbb{P}$. Assume that $p$ is compatible with $r$. 
Since the open set $O_0$ is dense, there is $s_{p,0} \in C_0$ such that $s_{p,0}(0) \leq p$ and $[s_{p,0}] \subseteq O_0$. Let 
\[
\varphi^0(p):= 
\begin{cases}
s_{p,0}(\text{end}),& \text{if $p$ is compatible with  $r$},\\
p,&\text{ otherwise.}
\end{cases}
\]
Fix $n \in \w$ and $t \in C_n$.
Define $\varphi^{n+1}_t\colon M \cap \mathbb P \rightarrow M \cap \mathbb P$ as follows: 
Fix $p\in P$.
Assume that $p$ is compatible with $t(\text{end})$.
Since the open set $O_{n+1}$ is dense, there is a finite nonempty sequence of conditions $s_{p,n,t}$ such that 
\[
s_{p,n,t}(0) \leq p, t(\text{end})\quad\text{ and }\quad[t\bigvid s_{p,n,t}] \subseteq O_{n+1}.
\]
Let
\[
\varphi^{n+1}_t(p) := 
\begin{cases}
s_{p,n,t}(\text{end}),& \text{if $p$ is compatible with  $t(\text{end})$},\\
p,&\text{ otherwise.}
\end{cases}
\]

By property $(\dagger)$, there is an $(M, \mathbb{P})$-generic condition $q \leq r$   such that 
\[
q \forces "\Gamma_\mathbb{P} \cap \varphi^0[M \cap \mathbb P] \neq \emptyset``,\quad\text{ and }\quad  q \forces "\Gamma_\mathbb{P} \cap \varphi^n_t[M \cap \mathbb P] \neq \emptyset``
\]
for all $n \geq 1$ and $t \in C_n$. 

By density, we can find some $q \in G$ forcing the above.
From now on we work in $V[G]$.
Let us define $R:= \bigcap_{n \in \w} U_n$, where $U_n := \bigcup \sset{[s]}{s \in C, [s] \cap V \subseteq O_n}$.
The set $R$ is coded in $V$ and $U_n = O_n \cap V$ for all $n \in \w$. 
We interpret the definitions of $Y$ and $f$ in the generic extension and obtain in the same way as before that $Y$ is homeomorphic to $\w^\w$ and $f \colon Y \rightarrow 2^\w$ is continuous and nowhere constant. Both $Y$ and $f$ are coded in $V$.  
We will now inductively define some $y \in Y$. 
Pick $q_0 \in G \cap \varphi^0[M \cap \mathbb P]$.
Then there exist $r_0$ and $s_{r_0,0}$ with $[s_{r_0,0}] \sub U_0$ and $q_0 = s_{r_0,0}(\text{end})$.
In particular, for each $j \in \mathrm{dom}(s_{r_0,0})$ we have $s_{r_0,0}(j) \in G$.
Let $y_0 := s_0 := s_{r_0,0}$.
Suppose we have already constructed $y_n = s_0\bigvid s_1 \bigvid \dotsb \bigvid s_n$ such that $[s_0\bigvid \dotsb \bigvid s_k] \subseteq U_k$ holds for all $k \leq n$ and $y_n(j) \in G$ for all $j \in \mathrm{dom}(y_n)$.
We have $y_n \in \bigcup_{k \in \w} Y \restriction~k = \bigcup_{k \in \w} (Y \cap V) \restriction k$ and therefore, we can pick $q_{n+1} \in G \cap \varphi_{y_n}^{n+1}[M \cap \mathbb{P}]$.
Then there is some $s_{n+1}$ with $[y_n\bigvid s_{n+1}] \subseteq U_{n+1}$ and $s_{n+1}(\text{end}) = q_{n+1}$. Again we have $s_{n+1}(j) \in G$ for all $j$.
Let $y_{n+1} := y_n\bigvid s_{n+1}$.
Finally, define $y:= \bigcup_{n \in \w} y_n \in Y$. 

By the definition of $Y$, we have that for all $i \in \w$ there is some $t_i \in 2^i$ such that $y(i) \forces "\tau \restriction i = t_i``$.
Hence, $\tau^G = f(y)$.
Furthermore, by the construction, we have that $y \in [y_n] \sub U_n$ for all $n \in \w$.
It follows that $y \in \bigcap_{n \in \w} U_n$, and thus 
\[
\tau^G = f(y) \in f[\bigcap_{n \in \w} U_n].
\]

Recall that the set $(\bigcap_{n \in \w} U_n) \cap V = \bigcap_{n \in \w} (U_n \cap V) = \bigcap_{n \in \w} O_n$ is a witness for $X \cap V$ being universally meager, i.e., 
\[
(\bigcap_{n \in \w} U_n) \cap V \cap f^{-1}[(X \cap V) \cap f[Y \cap V]] = \emptyset.
\]
Therefore, $f[\bigcap_{n \in \w} U_n] \cap X \cap V = \emptyset$. By our assumptions on $V$ and $\alpha$, we obtain $f[\bigcap_{n \in \w} U_n] \cap X = \emptyset$, which is a contradiction to $\tau^G =f(y) \in f[\bigcap_{n \in \w} U_n] \cap X.$
\epf
 According to \cite[Proposition~2.3]{UM1} every totally imperfect Hurewicz subspace
 is universally meager. 
  Combined with Theorem~\ref{main_result}
this gives the following fact.

\begin{cor}\label{cor_dagger1}
Suppose that  $\la\mathbb{P}_\alpha, \dot{\mathbb{Q}}_\beta : \alpha \leq \w_2, \beta < \w_2 \ra$ is an iterated forcing construction with countable supports such that
\[
\mathbbm{1}_{\mathbb{P}_\alpha} \forces "|\dot{\mathbb{Q}}_\alpha|\leq\w_1\text{
and }\dot{\mathbb{Q}}_\alpha\text{ satisfies }  (\dagger)``
\]
 for all $\alpha < \w_2$.
Then in $V^{\IP_{\w_2}}$ all Hurewicz totally imperfect spaces  have size at most $\w_1$.
\end{cor}

Since the Miller forcing satisfies
$(\dagger)$~\cite[Lemma~2.5]{RepZd2019}, we have the following corollary answering \cite[Question~1.10]{Zdo20}.

\begin{cor}\label{cor_miller1}
In the Miller model, every Hurewicz totally imperfect space has size at most $\w_1$.
\end{cor}

By \cite[Theorems 9 and 11]{perfectlymeager2002} (see also 
\cite[Theorem 17]{perfectlymeager2002} for a more general result)
every perfectly meager subspace is universally meager 
in the Miller model. 

\begin{cor}\label{cor_miller1_pm}
In the Miller model, every perfectly meager subspace has size at most $\w_1$.
\end{cor}

A \emph{Luzin set} is an uncountable subset of $\Cantor$ which has countable intersection with every meager subset of $\Cantor$.
Let us note that by
\cite[P. 529, Theorem~2]{Kur66} attributed there to Luzin,
every Luzin set admits a one-to-one map onto a perfectly meager
subspace, hence in the Cohen model there exists a perfectly meager subspace of size $\w_2=\hot c$, since the ``standard''
Cohen reals form a Luzin set of size $\w_2$, in this model. 
On the other hand, in this model there is no universally meager subspace 
of size $\w_2$ by \cite[p.~577, Theorem]{Mil83}, see the discussion after Problem~\ref{prob:6.5} for more details. Thus, the non-existence of universally meager subspaces of size $\hot c$ does not imply that also all perfectly meager subspaces have size less than $\hot c$.


\section{Rothberger spaces in some models of $\mathfrak{u} < \mathfrak{g}$}

The smallest cardinality $|\mathcal B|$ of a base 
$\mathcal B$ of an ultrafilter on $\w$ is denoted by $\hot u$.
For sets $Y$ and $Y'$, we write $Y'\as Y$, if $|Y'\setminus Y|<\w$.
A family $A\sub [\w]^\w$ is called \emph{groupwise dense}
if $Y\in A$ and $Y'\as Y$, where $Y'\in\roth$, implies $Y'\in A$, and for every sequence
$\la K_n:n\in\w\ra$ of mutually disjoint non-empty finite subsets of $\w$ there exists $I\in[\w]^\w$ with $\bigcup_{n\in I}K_n\in A$.
The smallest cardinality of a collection $\mathcal A$
of groupwise dense families with $\bigcap \mathcal A=\emptyset$
is denoted by $\fg$.
It is known that $\w_1\leq \hot u$ and $\hot g\leq\hot c$ (see~\cite{Bla10} for more information on these and other 
cardinal characteristics of the continuum).

The inequality $\hot u<\hot g$ 
is known to be consistent \cite{BlassLaflamme} and to have several consequences 
for combinatorial covering properties. 
For instance, it implies that every Rothberger space is 
Hurewicz, see \cite[Theorem~5]{SemifilterZdo} or 
\cite[Theorems 2.2 and 3.1]{TsaZdo08}. It is folklore 
that Rothberger spaces are totally imperfect because
the Rothberger property is preserved by closed subspaces and 
the Cantor space $2^\w$ is not Rothberger,  see, e.g., \cite[Theorem~2.3]{coc2}
for a considerably stronger result. Thus,
$\hot u<\hot g$ implies that Rothberger spaces are Hurewicz totally imperfect,
and hence the next fact follows directly from 
Corollary~\ref{cor_dagger1}.

\begin{cor}\label{cor_dagger2}
Let $\IP$ be the limit of an iteration like in Corollary~\ref{cor_dagger1}
such that $V^{\IP}\vDash(\hot u<\hot g)$. Then in 
$V^{\IP}$, every Rothberger space has size at most $\w_1$.
\end{cor}

Since $\hot u<\hot g$ holds  in the Miller model
(this is well-known and has been mentioned in the proof of  \cite[Theorem~2]{BlassLaflamme}, 
below we prove a more general Lemma~\ref{u<g_in_gjs}),
we conclude that Theorem~\ref{main_miller} below 
is a direct consequence of Corollary~\ref{cor_dagger2} and from the above discussion.

\bthm\label{main_miller}
In the Miller model, every Rothberger space has size at most $\w_1$.
More precisely, in this model
every Hurewicz totally imperfect space has size 
at most $\w_1$ and every Rothberger space is Hurewicz totally imperfect.
\ethm

 Let us note that the condition $\hot u<\hot g$ cannot be 
omitted here because the  Cohen forcing satisfies $(\dagger)$,
and the 
$\w_2$ many Cohen generics $c_\alpha$ ``directly'' added by the iteration of Cohen forcing with countable supports form a Rothberger space. Indeed, 
the set $C=\{c_\alpha:\alpha<\w_2\}$ is $\w_2$-concentrated on any countable dense subset of $2^\w$ and $\mathrm{cov}(\mathcal M)=\w_2=\hot c$ in the corresponding forcing extension, so it suffices to apply\footnote{We do not know whether 
 $C$ is a Luzin set. } \cite[Theorem~5]{FreMil}.


Next, we shall analyze 
iterations of length $\w_2$ of proper forcings with countable supports
introduced  by 
Goldstern, Judah and Shelah  in \cite{GJS1993} where they studied  how to create
stong measure zero sets of size $\hot c>\w_1$ without adding Cohen reals. 
In what follows we shall refer to such an iteration as to a GJS-iteration.
 We shall show that
 iterands in a GJS-iteration (and therefore also the final poset) satisfy 
$(\dagger)$, and by a combination of several well-known results we also have that 
$\hot u<\hot g$ holds in the corresponding extensions.  Thus, 
GJS-iterations will be shown to be within the scope of Corollary~\ref{cor_dagger2}.

More precisely, we say that
 an iteration $\la\IP_\alpha,\dot{\IQ}_\alpha:\alpha < \w_2\ra$ is a \emph{GJS-iteration} if
for cofinally many $\alpha\in\w_2$ we have that
$\dot{\IQ}_\alpha$ is forced to be the Miller forcing;
the other iterands $\dot{\IQ}_\alpha$ are ($\IP_\alpha$-names for) posets of the form $\mathrm{PT}_H$ depending on a
function $H:\w\to ([\w]^{<\w}\setminus[\w]^{\leq 1})$ which is provided by some bookkeeping procedure making sure that for each such $H$ in the final model
the poset $\mathrm{PT}_H$ is used as $\IQ_\alpha$ for cofinally many $\alpha$.
Define $\mathrm{PT}_H$ as the set of all subtrees $p\in\w^{<\w}$ such that
\begin{enumerate}
\item $\forall \eta \in p \: \forall l \in \mathrm{dom}(\eta) \enskip (\eta(l) \in H(l)),$
\item $\forall \eta \in p \enskip (|\mathrm{succ}_p(\eta)| \in \{1,|H(|\eta|)|\}),$
\item $\forall \eta \in p \: \exists \nu \in p \enskip (\eta \subseteq \nu \text{ and } |\mathrm{succ}_p(\nu)| = |H(|\nu|)|).$
\end{enumerate}
For $p_1,p_2\in \mathrm{PT}_H$ let
$p_1\leq p_0$, i.e., $p_1$ is a stronger condition than $p_0$,
if and only if $p_1\sub p_0$.

 In \cite[Corollary~2.14]{GJS1993} it was shown that $\mathrm{PT}_H$ is a proper $\w^\w$-bounding forcing notion. 
 We shall prove that it also satisfies $(\dagger)$ following the approach used in \cite[Lemma~2.5]{RepZd2019}. For this we need
 to introduce some auxiliary notations describing the fusion of
 conditions in $\mathrm{PT}_H$.

For $p \in \mathrm{PT}_H$ the set of all splitting nodes is defined by 
\[
\Split(p) := \{\eta \in p : |\mathrm{succ}_p(\eta)| > 1\}.
\]
Moreover, 
$$\Split_k(p) := \{\eta \in \Split(p) : |\{\nu \subsetneq \eta: \nu \in\Split(p)\}| = k\}.$$
For $p,q \in \mathrm{PT}_H$ we write $p \leq_k q$ if $p \leq q$ and $\Split_k(q) \subseteq p$. If we have a sequence $\la p_k : k \in \w\}$ with $p_{k+1} \leq_k p_k$, let $p_\w := \bigcap_{k \in \w} p_k \in \mathrm{PT}_H$ be the fusion of this sequence, which clearly satisfies $p_\w \leq_k p_k$ for all $k \in \w$. 

\blem \label{pth_dagger}
$\mathrm{PT}_H$ satisfies $(\dagger)$.
\elem
\bpf
Let $M \ni H$ and $\{\varphi_i: i \in \w\}$ be as in the formulation of $(\dagger)$. Suppose that $p \in M \cap \mathrm{PT}_H$ and $\{D_n: n \in \w\}$ is an enumeration of all open dense subsets of $\mathrm{PT}_H$ lying in  $ M$. Set $p_0 := p$. Suppose $\la p_n: n \leq k\ra$ has been constructed such that $p_n \leq_{n-1} p_{n-1}$ and for all $t \in \Split_n(p_n)$ we have $(p_n)_t \in  M$. 

For every $t \in \Split_k(p_k)$ and $m \in H(|t|)$, there is $R_{t,m} \in D_n \cap M$ such that $R_{t,m} \leq \varphi_k((p_k)_{t \bigvid m})$.
Let
\[
p_{k+1} := \bigcup \{R_{t,m}: t \in \Split_k(p_k), m \in H(|t|)\}.
\]
Note that $p_{k+1} \leq_k p_k$ and for all $r \in \Split_{k+1}(p_{k+1})$ there exist $t$ and $m$ such that $(p_{k+1})_r = R_{t,m} \in M$.
Let $p_\w := \bigcap_{k \in \w} p_k$ be the fusion. 
We have 
\[
p_\w \leq p_{k+1} \forces "\exists r \in \Split_{k+1}(p_{k+1}) \: \exists t,m \enskip ((p_{k+1})_r = R_{t,m} \in D_k \cap M \cap \Gamma``
\]
for all $k \in \w$.
In particular, $p_\w \forces "\exists t,m \enskip (\varphi_k((p_k)_{t \bigvid m}) \in \Gamma``$. 
Thus, $\mathrm{PT}_H$ satisfies $(\dagger)$. 
\epf

Lemma~\ref{pth_dagger}, Theorem~\ref{main_result}, the fact that the Miller forcing satisfies $(\dagger)$, \cite[Lemma~2.5]{RepZd2019}, and the definition of GJS-iterations imply the following fact.

\begin{cor}\label{gjs_sat_dagger}
Let $\la\IP_\alpha,\dot{\IQ}_\alpha:\alpha < \w_2\ra$ be a GJS-iteration. Then $\dot{\IQ}_\alpha$ is forced to satisfy 
$(\dagger)$ for all $\alpha<\w_2$, and hence also $\IP_{\w_2}$
satisfies $(\dagger)$.
\end{cor}

The next result may be thought of as a folklore.

\blem\label{u<g_in_gjs}
Let $\la\IP_\alpha,\dot{\IQ}_\alpha:\alpha < \w_2\ra$ be a GJS-iteration over a ground model $V$ of CH. Then $\hot u<\hot g$ holds in $V^{\IP_{\w_2}}$.
\elem
\bpf
Let $G$ be $\IP_{\w_2}$-generic filter over $V$.
The equality  $\hot u=\w_1$ holding in $V[G]$ follows from
the fact that every $P$-point $\mathcal U\in V$ is preserved by $\IP_{\w_2}$, i.e., it generates 
an ultrafilter in $V[G]$. This in its turn can be shown by  
combining  several well-known results. Firstly,
$P$-points are preserved by posets $\mathrm{PT}_H$ according to
\cite[Lemma~2.22]{GJS1993}. The preservation of 
$P$-points by the Miller forcing has been proven in 
\cite[Prop.~4.1]{Mil84}, see also \cite[Lemma~10]{BlassShelah1989}
for an alternative proof. Finally, countable support iterations of 
proper posets preserving $P$-points also preserve $P$-points by
\cite[Theorem~4.1]{BlaShe87}.

Next, we shall show that $V[G]$ also satisfies $\hot g=\w_2$.
For this we need two standard facts.

\begin{claim} \label{lem_w^w-bounded}
Let $\mathbb{R}$ be $\w^\w$-bounding and let $H$ be a $\bbR$-generic filter over $V$.
If $\la F_n : n \in \w \ra \in V[H]$ is a sequence of finite, non-empty, disjoint subsets of $\w$, then there exists a sequence $\la K_m : m \in \w \ra \in V$ of finite, non-empty, disjoint subsets of $\w$ such that for every $m \in \w$ there is some $n$ such that $F_n \subseteq K_m$.
\end{claim}
\bpf
Pick a subsequence $\la F_{n_m}: m \in \w\ra$ such that $\max F_{n_m} < \min F_{n_{m+1}}$.

In $V$, for every $f \in V \cap \w^{\uparrow \w}$ define 
\[
K^f_0:=[0,f(0))\quad\text{ and }\quad K^f_m := [f(m-1),f(m))
\]
for all $m>0$ and 
\[
K^f := \la K_m^f: m \in \w \ra \in V
\]

Let $g \in \w^{\uparrow \w}$ be a function in $V[H]$, such that $g(m)= \max_{F_{n_m}}$.
Note that for all $m \in \w$ we have $F_{n_m} \subseteq [g(m-1),g(m)]$.
Since $\mathbb{R}$ is $\w^\w$-bounding, there is a function $f \in \w^{\uparrow \w} \cap V$ with $g \leq f$. 
Define $h \in \w^{\uparrow \w} \cap V$ recursively by setting 
\[
h(0) := f(0)\quad\text{ and }\quad h(n+1) := f(h(n)+1)+1.
\]
We have
\[
h(n) < g(h(n)) < g(h(n)+1) \leq f(h(n)+1) < h(n+1)
\]
for all $n \in \w$
Finally, consider $K^h: = \la K^h_m: m \in \w \ra \in V$.
We have
\[
F_{n_{h(n)+1}} \subseteq [g(h(n)),g(h(n)+1)] \subseteq [h(m),h(m+1)) = K^h_m
\]
for all $m \in \w$.
\epf
\begin{claim} \label{prop_w^w-bounded}
Let $\mathbb{R}$ be $\w^\w$-bounding and let $H$ be a $\mathbb{R}$-generic filter.
If $A$ is groupwise dense in $V$, then the family 
\[
(\downarrow A)^{V[H]}:= \{b \in [\w]^\w: \exists a \in A \enskip (b \subseteq a)\}
\]
is groupwise dense in $V[H]$.
\end{claim}
\bpf
The set $(\downarrow A)^{V[H]}$ is clearly closed downwards and closed under finite modifications.
Let $\la F_n : n \in \w \ra \in V[H]$ be a sequence of finite, non-empty, disjoint subsets of $\w$.
By Claim~\ref{lem_w^w-bounded} pick the corresponding $\la K_m : m \in \w \ra \in V$ such that for all $m \in \w$ there is $n_m \in \w$ with $F_{n_m} \subseteq K_m$.
Since $A$ is groupwise dense, there is a set $I \in [\w]^\w$ such that $\bigcup_{m \in I} K_m \in A$.
Let $J := \{n_m: m \in I\} \in [\w]^\w$.
We have $\bigcup_{j \in J} F_j \subseteq \bigcup_{m \in I} K_m$, and thus $\bigcup_{j \in J} F_j \in (\downarrow A)^{V[H]}$. 
\epf

The fact that the  Miller forcing can be used to enlarge $\hot g$
has been noted at the end of the proof of 
\cite[Theorem~2]{BlassLaflamme}. For our purposes it will be 
convenient to use the following direct consequence of 
\cite[Lemma~7]{BlassShelah1989}. In what follows we shall
denote by $\miller$ the Miller generic real, i.e.,
$\miller=\bigcup\bigcap H$, where $H$ is an $\mathbb M$-generic filter.

\begin{claim}\label{lem_Lemma_7}
For each $p_0\in\mathbb M$ there exist $h\in \w^{\uparrow\w}$
and $p_1\leq p_0$ such that for every $I\in [\w]^\w$ there exists 
$q\leq p_1$ forcing 
\[
\rng(\miller)\sub\bigcup_{n\in I}[h(n),h(n+1)).
\]
\end{claim}

\begin{claim}\label{lem_g_Miller}
Let $A\sub\roth$ be a groupwise dense family in $V$ and  $H$ be an $\bbM$-generic filter over $V$.
Then in $V[H]$, we have 
\[
\rng(\miller)\sub(\downarrow A)^{V[H]}.
\]
\end{claim}
\bpf
It suffices to show that the set 
$$D := \{q \in \mathbb{M}: q \forces "\exists Y \in A\enskip (\rng(\dot{\miller}) \subseteq Y)``\}$$
is dense in $\mathbb{M}$. 

Fix $p_0 \in \mathbb{M}$. 
Let $p_1\leq p_0$ and $h$ be such as in
Claim~\ref{lem_Lemma_7}.
Since $A$ is groupwise dense,
there exists a set $I\in [\w]^\w$ such that 
\[
Y:=\bigcup_{n\in I}[h(n),h(n+1))\in A.
\]
By Claim~\ref{lem_Lemma_7}, there is
$q\leq p_1$ forcing $\rng(\dot{\miller})\sub Y$.
Thus, $q\in D$.
\epf

We are in a position now to complete the proof of $\hot g=\w_2$
in $V[G]$ (recall that $G$ is a $\IP_{\w_2}$-generic filter over $V$).
 We work in $V[G]$. Let $\{A_\gamma:\gamma<\w_1\}$ be a family of 
 groupwise dense sets.
A standard argument (see, e.g., the proof of 
\cite[Theorem~2]{BlassLaflamme}) yields
an $\w_1$-CLUB $C \subseteq \w_2$ such that
 $$A \cap V[G_\alpha] \in V[G_\alpha] \text{ and } A \cap V[G_\alpha] \text{ is groupwise dense in } V[G_\alpha]$$
 for all $\alpha \in C$ and $\gamma\in\w_1$. 

Let $\beta$ be the minimal ordinal  greater than or equal to $\alpha$ for which $\dot{\mathbb{Q}}_\beta = \dot{\mathbb{M}}$.
The forcing $\mathbb R:=\mathbb{P}_{[\alpha,\beta)}^{G_\alpha}$ is  $\w^\w$-bounding.
Let $H$ be an $\mathbb R$-generic filter over $V[G_\alpha]$ such that $V[G_\beta]=V[G_\alpha*H]$ and $\miller_\beta$ be the generic Miller real added at stage  $\beta$ of the iteration.
By Claim~\ref{prop_w^w-bounded}, the set  
\[
A'_\gamma:=\big(\downarrow (A_\gamma \cap V[G_\alpha])\big)^{V[G_\alpha*H]}
\]
is groupwise dense in $V[G_\beta]=V[G_\alpha*H]$
for all $\gamma\in\w_1$.
By Claim~\ref{lem_g_Miller}, 
\[
\rng(\miller_\beta) \sub \bigcap_{\gamma < \w_1} \big(\downarrow (A_\gamma \cap V[G_\alpha])\big)^{V[G_\beta]} \subseteq \bigcap_{\gamma < \w_1} A_\gamma,
\]
and therefore $\bigcap_{\gamma < \w_1} A_\gamma \neq \emptyset$
in $V[G]$.
Thus, $\mathfrak{g} = \w_2$. 
\epf

\brem
If we replace 
$\mathrm{PT}_{\dot{H}}$ in a GJS-iteration with an arbitrary proper forcing of size at most continuum but in addition  demand that for stationary many $\alpha < \w_2$ of cofinality $\w_1$ the iterand $\dot{\mathbb{Q}}_\alpha$ is the Miller forcing, we also obtain that $\hot g=\w_2$ in the the resulting model.
It follows from the fact that $\beta=\alpha$ in the final part of the proof of Lemma~\ref{u<g_in_gjs}, which would make the use of  Claim~\ref{prop_w^w-bounded} obsolete. \hfill $\Box$
\erem

Finally, Theorem~\ref{main_gjs}, which is  formulated below in a more detailed form  for the convenience of the reader, is a direct consequence of 
Corollary~\ref{gjs_sat_dagger},  Lemma~\ref{u<g_in_gjs}
and Corollary~\ref{cor_dagger2}.

\bthm\label{main_gjs_with_Hurewicz} 
In a forcing extension of a ground model of CH obtained by a GJS-iteration, there exists a strong measure zero set $X \subseteq 2^\w$ of size $\w_2 = \mathfrak{c}$ but  every Rothberger space has size at most $ \w_1$.

More precisely, in such forcing extensions
every Hurewicz totally imperfect space has size 
at most $\w_1$ and every Rothberger space is Hurewicz and totally imperfect.
\ethm

Theorem~\ref{main_gjs_with_Hurewicz} yields that in the Goldstern--Judah--Shelah model, every Rothberger space is Hurewicz and totally imperfect.
It turns out that using known facts about the Goldstern--Judah--Shelah model, also the converse implication is true, in this model.

\bthm\label{thm:gjs_R=H}
In the Goldstern--Judah--Shelah model, a space is Rothberger if and only if it is Hurewicz and totally imperfect.
\ethm

\bpf
($\Rightarrow$)
Apply Theorem~\ref{main_gjs_with_Hurewicz}.

($\Leftarrow$)
Every Hurewicz totally imperfect subspace of $\Cantor$ is universally meager~\cite[Proposition~2.3]{UM1}.
In the Goldstern--Judah--Shelah model, by Corollary \ref{cor_dagger1} and Corollary \ref{gjs_sat_dagger} every universally meager subspace has size at most $\w_1$ and each set of size $\w_1$ is SMZ since in the proof of \cite[Theorem~3.8]{GJS1993} it is established that the ideal of strong measure zero sets is closed under unions of less than $\mathfrak{c}$ many sets in this model.
Every SMZ set with the Hurewicz property is Rothberger~\cite[Theorem~8]{FreMil}.
\epf

In the proof of Theorem \ref{thm:gjs_R=H} we have actually established the following:

\begin{prop} \label{prop_univmeager_SMZ}
If every universally meager space is SMZ, then every Hurewicz totally imperfect space is Rothberger.
\end{prop}


\section{Strong measure zero sets in the Miller model}

In this section we prove Theorem~\ref{main_smz}.

By a \emph{Miller tree} we understand a subtree $T$ of $\w^{<\w}$
consisting of increasing  finite sequences such that the following
conditions are satisfied:
\begin{itemize}
\item
 Every $t \in T$ has an extension $s\in T$ which
is splitting in $T$, i.e., there are more than one immediate
successors of $s$ in $T$;
\item If $s$ is splitting in $T$, then it
has infinitely many immediate successors in $T$.
\end{itemize}
 The \emph{Miller forcing} is the
collection $\mathbb M$ of all Miller trees ordered by inclusion,
i.e., smaller trees carry more information about the generic. This
poset was introduced in \cite{Mil84}.
For a Miller tree $T$ we
shall denote  the set of all splitting nodes of $T$ by $\spl(T)$.
Given $p\in\mathbb M$ and $s\in\Split(p)$, we denote by
$\mathrm{scsr}_p(s)$ the set of all immediate successors of $s$ in $\Split(p)$.
Let $\mathbb M_{\alpha}$ be the iteration with countable supports of
the Miller forcing  $\mathbb M$ of length $\alpha$.

We shall need the following straightforward  fact.

\blem\label{Miller_with_numbers}
For  every  $\mathbb M$-name $\tau$ for a real and $p_0\in\mathbb M$ there exists 
$p\leq_0 p_0$ with the following property: For every  $s\in\Split (p)$ there exists $u_s\in 2^\w$
such that 
\begin{itemize}
\item for every $n\in\w$ and
 all but finitely many  $t\in\scsr_p(s)$  we have
 $ p_t\forces"\tau\uhr n=u_s\uhr n``$; and
\item
$p_s \forces "\tau\uhr(\max(\rng(s))+1)=u_s\uhr(\max(\rng(s))+1)``$
for every $s\in\Split(p)$.
\end{itemize}
\elem
\bpf
Achieving the first item is rather easy. 
Regarding the second one, 
by removing finitely many immediate successors of each splitting node of
$p$, if necessary, we may assume that 
\[
p_t \forces "\tau\uhr(\max(\rng(s))+1)=u_s\uhr(\max(\rng(s))+1)``
\]
for every $s\in\Split (p)$ and $t\in\scsr_p(s)$.
Since the set $\{p_t:t\in \scsr_p(s)\}$ is predense below $p_s$, we have
\[
p_s \forces "\tau\uhr(\max(\rng(s))+1)=u_s\uhr(\max(\rng(s))+1)``.\qedhere
\]
\epf

For $A\in [\w]^\w$  a map $\psi:Y\to 2^A$ is said to be \emph{$=^*$-surjective}
if for every $x\in 2^A$ there exists $y\in Y$ such that $\psi(y)=^*x$,
i.e., $\psi(y)(k)=x(k)$ for all but finitely many $k\in A$.
The proof of the next Lemma is rather similar to that of \cite[Lemma~10]{JudShe94} and hence it might be 
thought of as a folklore.

\blem\label{huge_meager}
Let $Y\in V$ be a non-meager subset of $2^\w$.
Let $G$ be an $\mathbb M$-generic filter over $V$ and $m$ be the generic Miller real
added by $G$.
In $V[G]$ let  
$\psi\colon 2^\w\to 2^{\rng(m)}$ be the map assigning to each $y\in 2^\w$ its restriction
$y\uhr\rng(m)$.
Then the map $\psi\uhr Y$ is $=^*$-surjective. 
\elem
\bpf
Given a name $\tau$  for a real and $p_0\in\mathbb M$,
we need to find $r\leq p_0$ and $y\in Y$ such that 
\[
r\forces "y\uhr\rng(\dot{m})=^*\tau\uhr\rng(\dot{m})``,
\]
where $\dot{m}$ is a name for $m$.
For this aim, let us first consider
$p\leq_0 p_0$ and $\{u_s:s\in\Split(p)\}$
such as in Lemma~\ref{Miller_with_numbers}. Passing to an infinite subset of 
$\mathrm{scsr}_p(s)$, if necessary, we may in addition 
assume that 
$$(\rng(t_1)\setminus\rng(s))\cap (\rng(t_0)\setminus\rng(s))=\emptyset$$
for any $s\in\Split(p)$ and different  $t_0,t_1\in\mathrm{scsr}_p(s)$. 
Shrinking  $\scsr(s)$ for every $s\in\Split(p)$ even more, we may assume that
\begin{equation}\label{verz_rare_miller}
\big(\rng(t_1)\setminus\rng(s_1)\big)\cap \big(\rng(t_0)\setminus\rng(s_0)\big)=\emptyset
\end{equation}
for every different $s_0,s_1\in\Split(p)$ and 
$t_0\in\scsr(s_0)$, $t_1\in\scsr(s_1)$.
It follows that 
\[
A_s:=\{\rng(t)\setminus\rng(s)\, :\, t\in \scsr_p(s)\}
\]
is an infinite disjoint family of finite nonempty subsets of $\w$,
and hence 
$$M_s:=\big\{x\in 2^\w\, :\, \forall^*\, a\in A_s\, (x\uhr a \neq u_{t(a)}\uhr a )\big\}$$ 
is a meager subset of $2^\w$, where 
$t(a)\in\scsr_p(s)$ is such that 
\[
a=\rng(t(a))\setminus\rng(s)
\]
for all $a\in A_s$.
Thus there exists $y\in Y\setminus\bigcup_{s\in\Split(p)}M_s$, and
hence 
$$T_s:=\{t\in\scsr_p(s)\,:\, y\uhr a = u_{t(a)}\uhr a\}$$
is infinite for all $s\in \Split(p)$.
Let $r\leq_0 p$ be a condition such that if $s\in\Split(p)$
belongs to $r$, then $\scsr_r(s)=T_s$. We claim that 
\begin{equation}\label{eq_meager_miller}
r\forces  "y\uhr \big(\rng(\dot{m})\setminus \rng(s_0)\big) = \tau\uhr \big(\rng(\dot{m})\setminus \rng(s_0)\big)``,
\end{equation}
and thus
$r$ and $y$ are as required, where $s_0$ is the stem of $r$. 

Indeed, otherwise we can find $r_1\leq r$ and $k\in \w$
such that 
\begin{equation} \label{if_not_miller}
r_1 \forces "\big(k\in \rng(\dot{m})\setminus \rng(s_0)\big) \wedge
\big(y(k)\neq\tau(k)\big)``.
\end{equation}
By (\ref{verz_rare_miller})
the only way a condition stronger than $r$ (being stronger than $p$)
may decide that $k\in \rng(\dot{m})\setminus \rng(s_0)$ is that
 there must exist 
$s\in \Split(r)$ and $t\in\scsr_r(s)$ such that
$k\in \rng(t)\setminus \rng(s)$ and $r_1\leq r_t$.
Let $a:=\rng(t)\setminus \rng(s)$.
Then, by the definition of 
$r$, we have $t=t(a)\in T_s$, and thus
\begin{equation} \label{if_not1}
y(k)=u_{t(a)}(k).
\end{equation}
On the other hand, by our choice of $p$, according to Lemma~\ref{Miller_with_numbers}, we have that
\[
p_t \forces "\tau\uhr(\max(\rng(t))+1)=u_t\uhr(\max(\rng(t))+1)``,
\]
Since $a\sub \max(\rng(t))+1$, we have $p_t\forces"\tau\uhr a=u_{t(a)}\uhr a``$.
Since $k\in a$ and $r_t\leq p_t$, in particular we have
\begin{equation} \label{if_not2}
 r_t\forces"\tau(k)=u_{t(a)}(k)``.
\end{equation}
It remains to note that the conjunction of~(\ref{if_not1}) and~(\ref{if_not2})
contradicts~(\ref{if_not_miller}), which completes our proof.
\epf

Whenever we speak about uniformly continuous maps between subspaces of $2^A$ 
for some $A\in [\w]^\w$, we consider such $2^A$ together
with the ultrametric $\rho_A$ defined as follows:
If $y_0,y_1$ are different elements of $2^A$, then 
$$\rho_A(y_0,y_1)=2^{- \min\{l\in A\: :\: y_0(l)\neq y_1(l)\}}.$$

\begin{cor} \label{unif_cont_non-meag}
Let $Y\in V$ be a non-meager subset of $2^\w$ 
and $G$ be an $\mathbb M$-generic filter over $V$.
Then in $V[G]$
there is a countable family $\Phi$ of  uniformly continuous maps
from $Y$ to $2^\w$ such that $2^\w=\bigcup_{\varphi\in\Phi}\varphi[Y]$.
\end{cor}
\bpf
Throughout the proof we work in $V[G]$.
Let  $m$ be the  Miller real
added by $G$  and 
$\psi\colon 2^\w\to 2^{\rng(m)}$ be the map defined in Lemma~\ref{huge_meager}.
Then $\psi$ is uniformly continuous with respect to
$\rho_{\w}$ and $\rho_{\rng(m)}$,
because $2^\w$ is compact.

Let also $\theta\colon\w\to \rng(m)$ be the order preserving bijection.
Then the  map $\tilde{\theta}$ assigning to each 
$x\in 2^{\rng(m)}$ the composition $x\circ \theta\in 2^\w$
is a homeomorphism between $2^{\rng(m)}$ and $2^\w$, and
hence it is uniformly continuous, 
by the compactness of $2^{\rng(m)}$. 

For every $k\in\w$ and $s\in 2^k$ consider the map
$\nu_s\colon 2^{\w}\to 2^\w$ such that
\[
\nu_s(z)=s\cup\big(z\uhr (\w\setminus k)\big).
\]
Again, it is uniformly continuous, because
 $2^\w$ is compact.
 By Lemma~\ref{huge_meager}, the map $\psi\uhr Y\colon Y\to 2^{\rng(m)}$ is $=^*$-surjective, and thus the map
$\varphi=\tilde{\theta}\circ(\psi\uhr Y)\colon Y\to 2^\w$ is $=^*$-surjective, too.
Then the set of all maps
\[
\nu_s\circ \tilde{\theta}\circ (\psi\uhr Y)\colon Y\to 2^\w
\]
for $s\in 2^{<\w}$ is as required.
\epf

The next easy fact is probably well-known, we present a streamlined proof
for the sake of completeness. 

\blem \label{laver_propert}
If $\IP$ has the Laver property, then in $V^{\IP}$ the ground model reals
$2^\w\cap V$ are not SMZ.
\elem
\bpf
Let $G$ be a $\IP$-generic filter over $V$.
In $V[G]$, fix $\la s_n:n\in\w\ra\in\prod_{n\in\w}2^{(n+1)^2}$.
By the Laver property of $\IP$, there is a sequence
\[
\la S_n:n\in\w\ra\in\prod_{n\in\w}[2^{(n+1)^2}]^{n}\cap V
\]
such that $s_n\in S_n$ for all $n\in\w$.
Since $(n+1)^2-n^2=2n+1>n$ for all $n\in\w$, in $V$ we can 
construct a sequence  
\[
\la t_n:n\in\w\ra\in\prod_{n\in\w}2^{[n^2,(n+1)^2)}
\]
such that $t_n\neq s\uhr [n^2,(n+1)^2)$ for all $s\in S_n$ and $n\in\w$.
Let $x:=\bigcup_{n\in\w}t_n\in 2^\w\cap V$.
Then $x\not\in [s]$ for any $s\in\bigcup_{n\in\w}S_n$,
in particular $x\not\in [s_n]$ for any $n\in\w$, which completes our proof.
\epf

\begin{cor} \label{almost_done_smz_miller}
In $V^{\mathbb M_{\w_2}}$ every SMZ set is meager.
\end{cor}
\bpf
Let $G$ be an $\mathbb M_{\w_2}$-generic.
In $V[G]$, fix a non-meager set $Y_0\sub 2^\w$.
Suppose, contrary to our claim, that   
$Y_0$ is SMZ in $V[G]$.
A standard argument yields 
$\alpha\in\w_2$ such that the set $Y:=Y_0\cap V[G_\alpha]\in V[G_\alpha]$
is non-meager in $V[G_\alpha]$.
By Corollary~\ref{unif_cont_non-meag}  
there exists a countable family $\Phi$ of uniformly continuous 
maps from $Y$ to $2^\w$ such that 
$2^\w\cap V[G_{\alpha+1}]=\bigcup_{\varphi\in\Phi}\varphi[Y]$.
Since SMZ sets are preserved by uniformly continuous maps and countable unions,
we conclude that $2^\w\cap V[G_{\alpha+1}]$ is SMZ in $V[G]$.
The forcing $\mathbb M$ has the Laver property~\cite[Theorem~7.3.45]{BarJud95}
and the Laver property is preserved by iterations of proper posets with countable supports~\cite[Theorem~6.3.34]{BarJud95}. 
It follows that $\bbM_{\w_2}$ has the Laver property.
By Lemma~\ref{laver_propert}, the set $2^\w\cap V[G_{\alpha+1}]$ is not SMZ in $V[G]$, a contradiction.
\epf

\bpf[{Proof of Theorem~\ref{main_smz}}] 
We work in the Miller model throughout the proof.
By Corollary~\ref{cor_miller1_pm} it suffices to 
show that all SMZ sets are perfectly meager.
Let $P\sub\Cantor$ be a perfect set and $X$ be a SMZ set.
Fix a homeomorphism $f\colon P\to\Cantor$.
Since $X\cap P$ is SMZ and $f$ is uniformly continuous,  
$f[X\cap P]$ is SMZ and hence it is meager in $2^\w$ by
Corollary~\ref{almost_done_smz_miller}.
Thus, the set $X\cap P$
is meager in $P$, which completes our proof. 
\epf


\section{Comments and open questions}

\subsection*{On property $(\dagger)$}
Recall that a real $g$ in a generic extension is called eventually different if for all $y \in \w^\w \cap V$ the set $\sset{n \in \w}{y(n) = g(n)}$ is finite. 
\begin{prop} \label{prop_eventdiff}
Let $\mathbb P$ be a forcing notion satisfying $(\dagger)$. Then in $V^{\mathbb{P}}$ there are no eventually different reals.
\end{prop}

\begin{proof}
Let $\tau$ be a name for a real and let $p \in \mathbb P$. Let $M \ni \mathbb P, \tau, p$ be a countable elementary submodel of $H(\theta)$ for sufficiently large $\theta$. Enumerate $M \cap \mathbb P = \sset{p_i}{i \in \w}$. Let $\varphi \colon M \cap \mathbb{P} \rightarrow M \cap \mathbb{P}$ be a function such that $\varphi(p_i) \leq p_i$ decides $\tau(i)$ to be $a_i \in \w$ for all $i \in \w$. Define $g_M \colon \w^\w \rightarrow \w^\w$ by setting $g_M(i) := a_i$. By property $(\dagger)$, there is $q \leq p$ such that 
\[
q \forces "\exists^\infty p \in M \cap \mathbb{P} \enskip (\varphi(p) \in \Gamma_\mathbb{P})``.
\]

Let $G \subseteq \mathbb{P}$ be a generic filter. By density, we can find $q \in \Gamma_\mathbb{P}$ like above. Then $\sset{i \in \w}{\tau(i) = a_i = g_M(i)}$ is infinite. 
\end{proof}

\begin{rem}
As a consequence of Proposition \ref{prop_eventdiff}, Random forcing, Hechler forcing and Laver forcing all do not satisfy property $(\dagger)$. Note that although there are ccc forcings violating property $(\dagger)$, it is easy to see that all $\sigma$-closed forcings satisfy property $(\dagger)$. Since Sacks forcing and $PT_H$ satisfy $(\dagger)$, one could conjecture that all strong Axiom A forcings, see \cite[Definition~1.3]{MildShelah2019}, satisfy property $(\dagger)$. However, the forcing $PT_{f,g}$, which was introduced in \cite[Definition~7.3.3]{BarJud95}, satisfies strong Axiom A \cite[Lemma~7.3.5]{BarJud95}, but does not satisfy property $(\dagger)$ since it adds eventually different reals \cite[Lemma~7.3.6]{BarJud95}. \hfill $\Box$
\end{rem}

Is is well known that adding eventually different reals is equivalent to $\w^\w \cap V$ being meager in $\w^\w$ in the generic extension, see \cite[Lemma~2.4.8]{BarJud95}. 

Let us note that in all our applications the inequality $\fu<\fg$ was essential. Maybe our result follows from it alone. An important case to check is the forcing introduced by Blass and Shelah in \cite{BlaShe87}, which was the first model for which $\fu<\fg$ was shown \cite[Theorem~2]{BlassLaflamme}. Here we have $\fc = \fs = \mathit{non}(\mathcal{M})$, see \cite[Theorem~5.2]{BlaShe87}, which implies that $\w^\w \cap V$ is meager and therefore, $(\dagger)$ cannot be fulfilled by the forcing. 
\bprb
Does it follow from $\fu<\fg$ that all universally meager subspaces (Hurewicz totally imperfect) have size at most $\w_1$?
What about the Blass--Shelah model?
\eprb

The next problem asks whether the main result of \cite{GJS1993} can be improved towards obtaining even Rothberger spaces of size $\mathfrak c$ without adding Cohen reals. Theorem~\ref{main_gjs_with_Hurewicz} implies that 
a conceptually different model is needed from the one constructed in \cite{GJS1993}. 

\bprb 
Can a forcing extension of a model of GCH contain a Rothberger space of size $\w_2$, if the corresponding poset does not add Cohen reals, and is obtained by an iteration of length $\w_2$ of proper posets of
size $\w_1$  with countable supports?
\eprb

\subsection*{Relations between Rothberger spaces and Hurewicz totally imperfect spaces}

We proved that in the Goldstern--Judah--Shelah model the classes of Rothberger spaces and Hurewicz totally imperfect spaces coincide.
Under CH, there is a Rothberger space which is not Hurewicz, e.g., a Luzin set, and a Hurewicz totally imperfect space which is not Rothberger, e.g., a Sierpiński set, see \cite{coc2, Scheepers96} and references therein.
It follows that under CH, Rothberger and Hurewicz totally imperfect are incomparable properties.

In the Laver model all Rothberger spaces are countable (and thus trivial) and there is in ZFC an uncountable Hurewicz totally imperfect space.
In the Miller model, there is an uncountable Rothberger space and Rothberger implies Hurewicz totally imperfect (it follows from the inequality $\fu<\fg$, in this model).

\bprb \label{prob:6.5}
Is there a totally imperfect Hurewicz space which is not Rothberger, in the Miller model?
\eprb

In the other direction, it is a folklore fact\footnote{We have learned it from P. Zakrzewski.} that in the model obtained by adding $\w_2$ Cohen reals with finite supports to a model of CH, every universally meager (hence also Hurewicz totally imperfect) subspace has size at most $\w_1$. Indeed, if $X\in [2^\w]^{\w_2}$ in this model, then there exists a 
Luzin set $L\sub 2^\w$ and an injective continuous map $f\colon L\to 2^\w$ such that $|f[L]\cap X|=\w_2$, see~\cite[p.~577]{Mil83}. Passing to a co-countable closed subspace of $L$, if necessary, we may additionally assume that 
$L$ has no isolated points. Let $R\supseteq L$ be a $G_\delta$-subspace of
$2^\w$ such that $L$ is dense in $R$ and $f$ can be extended to 
a continuous map $\bar{f}:R\to 2^\w$. 
Thus, $R$ is a Polish space, $\bar{f}$ is nowhere constant because 
its restriction to $L$, namely $f$, is injective, and $\bar{f}^{-1}[X]$ is non-meager in $R$ because it contains an uncountable subspace of $L$, and hence $X$ is not universally meager by \cite[Theorem~1.1]{UM2}.

Since adding $\w_1$ Cohen reals makes any ground model set of reals Rothberger, see \cite[Theorem~11]{ScheepersTall}, every Hurewicz totally imperfect space is Rothberger in this model. The inverse implication is not true since the Cohen reals form a Luzin set, hence are Rothberger, but not even meager.

Another model which should be mentioned in this context is the one obtained by Bartoszyński and Shelah in \cite{BS02} by the countable support iteration of length $\w_2$ over a model of CH of the poset $PT_H$ for a fixed strictly increasing $H$ of the form $H(n)=2^{h(n)}$ for some $h\in\w^\w$. Here again every Hurewicz totally imperfect space is Rothberger. Indeed, every universally meager subspace has size at most $\w_1$, see \cite[Section~4]{BS02}, and every subspace of size $\w_1$ is SMZ because even a union of any $\w_1$ many SMZ sets is SMZ, which can be proven by combining \cite[Fact~0.4]{GJS1993} with the following modification of \cite[Fact~2.26]{GJS1993}: 

\begin{fact}
Let $h^* \colon \w \rightarrow \w \setminus \{0\}$, $H^* \colon \w \rightarrow \mathcal{P}(2^{<\w})$ with $H^*(n) = 2^{h^*(n)}$, $\mathcal{H}$ be a dominating family, $\bar{\nu}$ have index $\mathcal{H}$ and $\dot{g}$ be a name for the generic function added by $PT_{H^*}$. Then 
\[
\begin{aligned}
\forces_{PT_{H^*}} &"\forall h \in \mathcal{H} \enskip \exists d \in [\w]^\w \enskip \exists h_1 \in \mathcal{H} \\ 
&\big( h \leq d \leq h_1 \text{ and } \bigcup_{k \in \w} [\nu^{h_1}(k)] \subseteq \bigcup_{n \in d} [\dot{g}(n)]\big)``. 
\end{aligned}
\]
\end{fact}
 Therefore, every universally meager subspace is SMZ in this model and thus, by Proposition \ref{prop_univmeager_SMZ}, every Hurewicz totally imperfect space is Rothberger. 

\bprb
Is there a Rothberger space which is not Hurewicz, in the model of Bartoszy\'{n}ski--Shelah?
\eprb

\bibliographystyle{abbrvurl}

\bibliography{library}

\end{document}